\documentclass[11pt]{amsart}
\usepackage{amssymb}
\usepackage{graphicx,color}
\usepackage{amsmath,amscd,amsthm,amssymb,amsxtra,latexsym,epsfig,graphics}
\usepackage[all]{xy}

\numberwithin{equation}{section} 
\newtheorem{Theorem}{Theorem}[section]

\newtheorem{Proposition}[Theorem]{Proposition}
\newtheorem{Corollary}[Theorem]{Corollary}

\theoremstyle{definition}
\newtheorem{Definition}[Theorem]{Definition}

\newtheorem{Remark}[Theorem]{Remark}

\newtheorem{Example}[Theorem]{Example}

\newcommand{\CC}{{\mathbb C}}

\newcommand{\ZZ}{{\mathbb Z}}
\newcommand{\NN}{{\mathbb N}}

\newcommand{\PP}{{\mathbb P}}

\newcommand{\TT}{{\mathcal T}}

\newcommand{\II}{{\mathcal I}}

\newcommand{\OO}{{\mathcal O}}

\newcommand{\rk}{\operatorname{rk}}

\newcommand{\Hom}{\operatorname{Hom}}

\newcommand{\Ext}{\operatorname{Ext}}

\makeindex

\begin{document}
\title[Nearly free curves and arrangements: a vector bundle point of view]
{Nearly free curves and arrangements: \\a vector bundle point of view}

\author[]{ S. Marchesi, J.\ Vall\`es}

\date{\today}

\begin{abstract} 
Many papers are devoted to study logarithmic sheaves associated to reduced divisors,
in particular logarithmic bundles associated to plane curves since forty years in differential and algebraic topology or geometry.
An interesting family of these curves are the so-called free ones for which the associated logarithmic sheaf is the direct
sum of two line bundles. When the curve is a finite set of distinct lines (i.e. a line arrangement),
Terao conjectured thirty years ago that its freeness depends only on its combinatorics.
A lot of efforts were done to prove it but at this time it is only proved up to 12 lines.
If one wants to find a counter example to this conjecture a new family of curves arises naturally: the nearly free curves introduced by Dimca and Sticlaru.
We prove here that the  logarithmic bundle associated to a nearly free curve possesses a minimal non zero section that vanishes on one single point $P$,
called jumping point, and that characterizes the bundle. Then we give a precise description of the behaviour of $P$.
In particular we show, based on detailed examples, that the position of $P$ relatively to its corresponding nearly free arrangement
of lines may or may not be a combinatorial invariant, depending on the chosen combinatorics. 
\end{abstract}

\maketitle 

\section{Introduction}

Given a reduced curve $C$ in $\PP^2=\PP^2(\CC)$ of degree $d$ defined as the zero locus of a homogeneous polynomial $f=0$, we define the Jacobian ideal of $f$, denoted by $\II_{\nabla f}$ as the image of the map
$$
\OO_{\PP^2} \stackrel{\nabla f}{\longrightarrow} \OO_{\PP^2}(d-1),
$$
where $\nabla f$ is the matrix whose entries are given by the partial derivatives of $f$ with respect to the three variables $x,y,z$, i.e. $\nabla f = [\frac{\partial f}{\partial x} \:\: \frac{\partial f}{\partial y} \:\: \frac{\partial f}{\partial z}]$. Its kernel $\TT_C$ is a rank two reflexive sheaf, therefore a vector bundle on $\PP^2$, defined by the following short exact sequence
$$
0 \longrightarrow \TT_C \longrightarrow \OO_{\PP^2}^3 \longrightarrow \II_{\nabla f}(d-1) \longrightarrow 0.
$$
In \cite{S} Saito introduced the notion of free divisor in affine and projective spaces of any dimension. In the same 
volume Terao studied arrangements of hyperplanes that are free divisors (see \cite{T}). In this paper we restrict our study to curves in the projective plane.  
\begin{Definition}\label{f-def}
A reduced curve $C\subset \PP^2$ is called \emph{free} with exponents $(a,b) \in \mathbb{N}^2$, with $a\leq b$ if the associated vector bundle $\TT_C$ is free i.e. if $\TT_C=\OO_{\PP^2}(-a) \oplus \OO_{\PP^2}(-b).$ 
\end{Definition}
Relatively few free curves are known. When the curve $C$ is a finite set of lines (i.e. a line arrangement) an important invariant attached to $C$ is its combinatorics. This combinatorics is described by an incidence graph of points and lines (for details we refer to \cite{OT} which is the reference book on the hyperplane arrangements). Probably the main conjecture about hyperplane arrangements, still open on any field and in any dimension $\ge 2$ is the so-called Terao's conjecture, stated in  
\cite[Conjecture 4.138]{OT}, which says in substance that the freeness of an arrangement depends only on its combinatorics. On the complex projective plane it is proved only up to $12$ lines (see \cite[Corollary 6.3]{FV}). 

\smallskip

If Terao's conjecture is not true then one can find two arrangements $C_0$ and $C$ with the same combinatorics such that $C_0$ is free but $C$ is not. In $\PP^2$ it implies in particular that $\TT_{C_0}=\OO_{\PP^2}(-a) \oplus \OO_{\PP^2}(-b)$ and 
$H^0(\TT_{C}(a-1))\neq 0$ (see \cite{FV}, Lemma 3.2). In particular, assuming here that $H^0(\TT_{C}(a-2))=0$, we obtain 
$$
0 \longrightarrow \OO_{\PP^2}(1-a) \longrightarrow \TT_{C} \longrightarrow \II_Z(-1-b) \longrightarrow 0,
$$
where $Z\subset \PP^2$ is a $0$-dimensional scheme of length $b-a+1$ (\cite[Lemmas 1 and 2]{B}. 

\smallskip

In order to explain the role of $Z$ let us recall some basic facts concerning the restriction of a rank vector bundle $E$
on a line.  By Grothendieck's theorem its restriction to any line $L$ splits as a sum of two line bundles $$E\otimes \OO_L=\OO_L(\alpha)\oplus \OO_L(\beta), \,\,\mathrm{with}\,\,\alpha+\beta=c_1(E).$$
The couple $(\alpha,\beta)\in \ZZ^2$ is called the \textit{splitting type of $E$ on $L$}; the positive integer 
$\delta_L(E):=|\alpha-\beta|$ is its \textit{gap}. It is known that, see \cite[page 29]{OSS} for a reference, this gap is minimal on an open set of the dual projective plane and we will denote its value by $\delta(E)$. The lines $L$ such that $\delta_L(E)>\delta(E)$ are the \textit{jumping lines of} $E$ and we  denote the set of jumping lines by $S(E)$.
The positive integer $o(L):=\frac{\delta_L(E)-\delta(E)}{2}$ is the \textit{order of the jumping line} $L$.

\smallskip

Now let us explain the key role of $Z$ in the context of Terao's conjecture. When a line arrangement $C$ has the same combinatorics than a free line arrangement $C_0$ but is not free, one associates to $C$ a $0$-dimensional scheme $Z$ characterizing its non-freeness in the following sense: since $C_0$ is free we have clearly $\delta_L(\TT_{C_0})=\delta(\TT_{C_0})$ for any $L$; on the contrary,  $\delta_L(\TT_{C})>\delta(\TT_{C})$ for any line $L$ meeting $Z$.   

\smallskip

The first case to be studied is naturally the case $Z=\{P\}$ is a single point (i.e. $a=b$). We will see in Theorem \ref{first-thm}
that the corresponding curve (or the corresponding bundle) is a \emph{nearly free curve} introduced by Dimca and Sticlaru in \cite{DS}.
\begin{Definition}\label{nf-def}
A reduced curve $C\subset \PP^2$ is called \emph{nearly free} with exponents $(a,b) \in \mathbb{N}^2$, with $a\leq b$ if the associated vector bundle $\TT_C$ has a resolution of type
\begin{equation}\label{NFdef}
0 \longrightarrow \OO_{\PP^2}(-b-1) \stackrel{M}{\longrightarrow} \OO_{\PP^2}(-a) \oplus \OO_{\PP^2}(-b)^2 \longrightarrow \TT_C \longrightarrow 0.
\end{equation}
\end{Definition}
\begin{Remark}
Dimca and Sticlaru  gave another definition and obtained this previous one as a characterization of nearly free curves 
(\cite[ Theorem 2.2]{DS}).
\end{Remark}
Throughout this paper, the vector bundle $\TT_C$, associated to a nearly free curve, will be called a \emph{nearly free vector bundle}. Following the seminal work of Dimca and Sticlaru \cite{DS} on nearly free divisors, many works were done and published and it is our belief that, in order to better understand nearly free divisors it is important to clarify what are the nearly free vector bundles. 

In particular, since this point $P$, called jumping point, characterizes the failure of freeness for a nearly free arrangement it is important to understand its behaviour relatively to the corresponding arrangement. This will be studied in section \ref{sec-NFVB}.

We will then focus, in the same section, on the study of the splitting type of a nearly free vector bundle, which has first been considered in the work of Abe and Dimca, see \cite{AD}. In their paper, the authors prove, among other results, that we only have two possible splitting types for a nearly free vector bundle (\cite[Theorem 5.3]{AD}). In this work, we retrieve their result, using the resolution of these vector bundles, and we complete it by describing the geometry of the set of jumping lines. Indeed we show in Proposition \ref{NFSplitting} that the locus of jumping line is the line $\check{P}$ of all lines through $P$ and that the order of any jumping line is $1$ (except of course for the tangent bundle that is nearly free but have no jumping line). Reciprocally we classify in Theorem \ref{characterization-by-jl} rank two vector bundles $E$ such that $S(E)$ is a line in the dual projective plane $\check{\PP}^2$ and such that the order of any jumping line is $1$. We show that this configuration of the jumping locus actually characterizes nearly free vector bundles in the unstable and semistable case. In  the stable case another very specific family of bundles arises, which we also classify completely,  this family does not concern directly our discussion because it is straightforward to notice that the only stable nearly free vector bundle with exponents $(a,b)$ is the tangent bundle twisted by $\OO_{\PP^2}(-b-1)$.

This description of the jumping locus allows us to answer some natural questions regarding the relation between the jumping lines of $\TT_C$ and the lines of the arrangement $C$ (see Corollary \ref{one-generic}). 

In section \ref{sec-addel}, we prove that each nearly free vector bundle can be seen as an extension of a line bundle on a line with a free vector bundle. This construction can be geometrically interpreted as the deletion of a line of the free arrangement passing through a specific amount of triple points (see Proposition \ref{PropDel}). We will also show that each nearly free bundle can be defined as the kernel of a surjective morphism between a free vector bundle and a line bundle on a line (see Proposition \ref{PropAdd}). This construction translates in the addition of a line to the free arrangement, passing again through a specific amount of triple points. If $N$ is the number of lines of the arrangement, $n$ the number of intersection points on one line of the arrangement and $t$ the number of triple points on the same line, then it is easy to verify that $t=N-n-1$. According to this equality the previous two propositions give another formulation of \cite[Theorem 5.7]{AD}.
Such two techniques suggest that we can construct any nearly free vector bundle by adding a line to a  properly chosen free arrangement and also by deleting one from a different free arrangement. In section \ref{sec-ex} we give explicit examples of these two constructions, for each nearly free vector bundle.

Finally, in section \ref{sec-jump} we prove that there is no explicit relation between the combinatorics of the arrangement and the jumping locus. Indeed, we provide examples to show that the jumping point of a nearly free vector bundle $\TT_C$ coming from an arrangement $C$ can be on exactly one line or on many lines of $C$ or outside $C$. Such examples show how the jumping points variates in the projective plane when shifting a line in order to maintain its combinatoric, see Example \ref{ex-1}. They also prove that in some case the combinatorics forces the jumping point to belong to one or multiple lines of the arrangement, see Example \ref{ex-line}; while for some other fixed combinatorics, see Example \ref{ex-inout}, the point can either belong or not to the arrangement.\\

\noindent\textbf{Acknowledgements}\\
The authors wish to thank Takuro Abe for fruitful discussions.

We would like to thank the University of Campinas and the Universit\'e de Pau et des Pays de l'Adour for the hospitality and for providing the best working conditions (in 2016 and 2018 in Campinas and 2017 in Pau).

The first author was partially supported by Funda\c{c}\~{a}o de Amparo \`{a} Pesquisa do Estado de S\~{a}o Paulo (FAPESP), grant 2017/03487-9, by the CNPq grant number 303075/2017-1 and by a grant CNRS(INSMI) given to LMAP. 

The second author was partially supported by Funda\c{c}\~{a}o de Amparo \`{a} Pesquisa do Estado de S\~{a}o Paulo (FAPESP), grant 2018/08524-2. 

\section{Nearly free arrangements and vector bundles}\label{sec-NFVB}
Directly from Definition \ref{nf-def} we have that $c_1(\TT_C) = -a -b +1$ and $c_2(\TT_C) = ab -a +1$. Moreover, it is possible to compute the first non vanishing degree of its global sections then, according to \cite[page 132]{B}, we can remark the following
\begin{itemize}
\item $\TT_C$ is stable (in the sense of Mumford-Takemoto) if and only if $a=b$, and in this case the bundle is the tangent bundle twisted by $\OO_{\PP^2}(-b-1)$, i.e. $\TT_C \simeq T_{\PP^2}(-b-1)$,
\item $\TT_C$ is semistable if and only if $a=b-1$,
\item $\TT_C$ is unstable if and only if $a<b-1$.
\end{itemize}
In any case, stable, semistable or unstable, we prove the following characterization of nearly free vector bundles or curves.
\begin{Theorem}\label{first-thm}
$\TT_C$ is nearly free with exponents $(a,b)\in \NN^2$, with $a\le b$, if and only if there exists a point $P\in \PP^2$ such that $\TT_C$ splits in the following exact sequence 
$$
0 \longrightarrow \OO_{\PP^2}(-a) \longrightarrow \TT_C \longrightarrow \II_P(-b+1)\longrightarrow 0.
$$ 
\end{Theorem}
\begin{proof}
 Let us consider a section $s \in H^0(\TT_C(a))$. Since $H^0(\TT_C(a-1))=0$, it defines the following short exact sequence (\cite[Lemmas 1 and 2]{B}
\begin{equation}\label{NFx}
0 \longrightarrow \OO_{\PP^2}(-a) \longrightarrow \TT_C \longrightarrow \II_Z(-b+1) \longrightarrow 0
\end{equation}
with $Z \subset \PP^2$ a 0-dimensional scheme of length  $c_2(\TT_C(a))=1$. In other words $Z$ is a point  $P\in \PP^2$ and we have actually the following commutative diagram
$$
\xymatrix{
&  & 0\ar[d] & 0\ar[d]\\
 & 0 \ar[r] \ar[d] & \OO_{\PP^2}(-a) \ar[d] \ar[r]^{\simeq} & \OO_{\PP^2}(-a) \ar[d] \ar[r] & 0\\
 0 \ar[r] & \OO_{\PP^2}(-b-1) \ar[d]^{\simeq} \ar[r] & \OO_{\PP^2}(-a) \oplus \OO_{\PP^2}(-b)^2 \ar[r] \ar[d] & \TT_C \ar[r] \ar[d] & 0 \\
 0\ar[r] & \OO_{\PP^2}(-b-1) \ar[d] \ar[r] & \OO_{\PP^2}(-b)^2 \ar[r] \ar[d] & \II_P(-b+1) \ar[d]\ar[r] & 0 \\
 &0 &0 &0 & & 
}
$$
This diagram is constructed considering an element of $\Hom(\OO_{\PP^2}(-a),\TT_C)$ which necessarily comes from an element of $\Hom(\OO_{\PP^2}(-a),\OO_{\PP^2}(-a)\oplus \OO_{\PP^2}(-b)^2)$, being\\ $\Ext^1(\OO_{\PP^2}(-a),\OO_{\PP^2}(-b-1))=0$, and we get the first two rows of the diagram. We complete using the Snake Lemma.
\end{proof}
\begin{Remark}
As it will be explained with more details in Proposition \ref{prop-detpoint}, a nearly free vector bundle is then completely determined by the data of its exponents $(a,b)$ and a point $P\in \PP^2$.
From now on, we will call the point $P$ the \emph{jumping point} of the nearly free vector bundle.
\end{Remark}
\begin{Example}
Consider the union of six lines through four non aligned points (see Figure \ref{con1}). This line arrangement is free with exponents $(2,3)$. It is well known and there are many ways to prove it; one of them consists in seeing these six lines as the three singular conics of a pencil of conics (see \cite{V},\cite[Theorem 2.7]{V2}).

\begin{figure}[ht!]
\begin{center}
\includegraphics[scale=0.35]{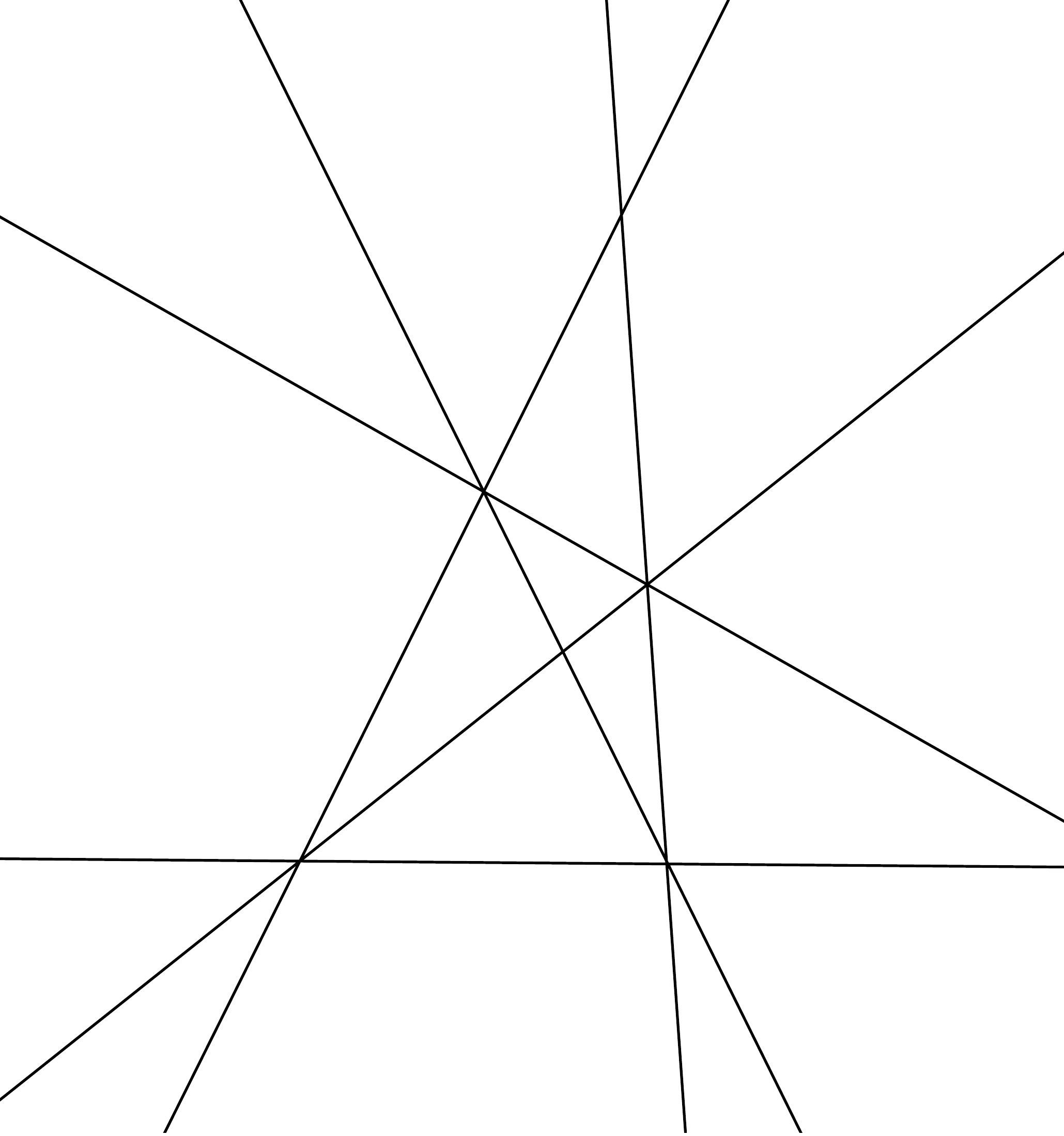}
\caption{}
\label{con1}
\end{center}
\end{figure}

Let us remove now a singular conic and replace it by a smooth one which lies in the pencil determined by the four singular points (see Figure \ref{con2}). Let $C$ be this curve formed by the union of two singular conics and one smooth conic of the same pencil. Then the associated logarithmic bundle $\TT_C$ is nearly free with exponents $(2,4)$ since it verifies 
$$
0 \longrightarrow \OO_{\PP^2}(-2) \longrightarrow \TT_C \longrightarrow \II_P(-3)\longrightarrow 0
$$ 
where the point $P$ is the singular point of the removed singular conic (see \cite{V}, \cite[Theorem 2.8]{V2} for details). 

\begin{figure}[ht!]
\begin{center}
\includegraphics[scale=0.35]{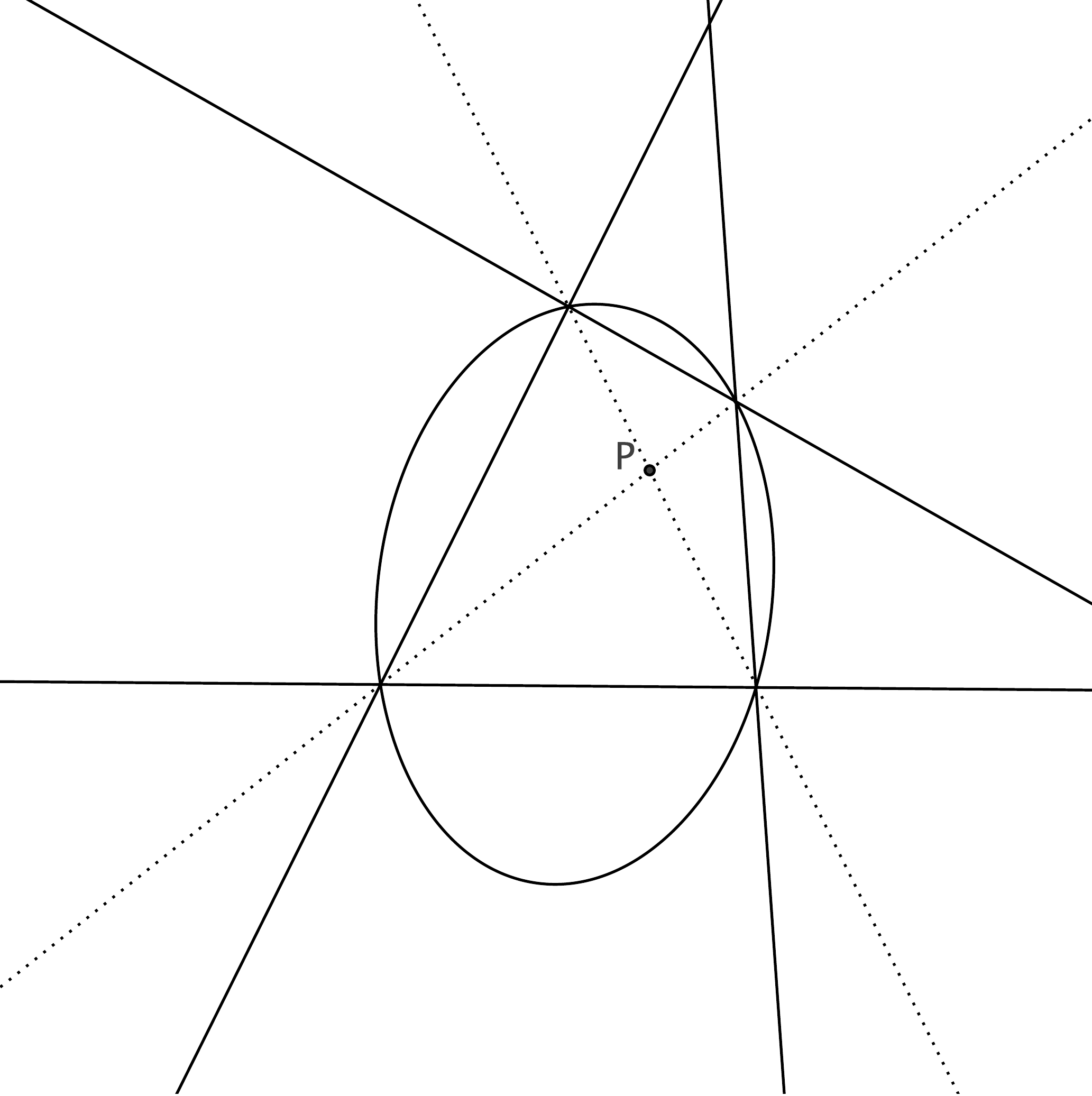}
\caption{}
\label{con2}
\end{center}
\end{figure}

\end{Example}
Through the previous description it is possible to prove the following result
\begin{Proposition}\label{NFSplitting}
Let $\TT_C$ be a nearly free vector bundle with exponents $(a,b)$, with $a \leq b$, and jumping point $P$. Then for every line $L\not\ni P$, we have ${\TT_C}_{|L} \simeq \OO_L(-a) \oplus \OO_L(-b+1)$ and for every line $L \ni P$ we have ${\TT_C}_{|L} \simeq \OO_L(-a+1) \oplus \OO_L(-b)$. Hence, we have only two possible splitting type. 
\end{Proposition}
\begin{Remark} Since the gap of the splitting type increases for lines through the jumping point $P$ these lines are called jumping lines of  $\TT_C$.
\end{Remark}
\begin{proof}
It comes directly considering the restriction of $\II_P(-b+1)$. Indeed, if $L \not \ni P$, $(\II_P(-b+1))_{|L} \simeq \OO_L(-b+1)$ and, being $b\geq a$, we have ${\TT_C}_{|L} \simeq \OO_L(-a) \oplus \OO_L(-b+1)$. If $L \ni P$, we get $(\II_P(-b+1))_{|L} \simeq \OO_L(-b)\oplus\OO_P$, implying that ${\TT_C}_{|L} \simeq \OO_L(-a+1) \oplus \OO_L(-b)$.
\end{proof}
Observe that the fact that a bundle associated to a nearly free curve has only two possible splitting types was already proved by Abe and Dimca in \cite[Theorem 5.3]{AD}. However they did not determine the set of jumping lines, and ask in \cite[Remark 2.6]{AD}  whether a line of the arrangement has the generic splitting type. As an immediate consequence of the previous result we answer their question.
\begin{Corollary}\label{one-generic}
Let $\TT_C$ be a nearly free vector bundle associated to an arrangement of lines. Then, at least one of the lines of the arrangement has the generic splitting type.
\end{Corollary}
\begin{proof}
The previous theorem tells us that a nearly free vector bundle is either stable with no jumping lines (it is the tangent bundle) or its jumping lines are the ones passing through a fixed point. It is well known that a vector bundle associated to an arrangement of $N$ lines through a fixed point is a free vector bundle with exponents $(0,N-1)$, therefore it cannot be our original $\TT_C$.
\end{proof}
Then we obtain a characterization of nearly free vector bundles.
\begin{Proposition}\label{prop-detpoint}
Given a point $P \in \PP^2$ and a couple of integers $(a,b)\in \NN^2$ with $ a\le b$, there exists, up to isomorphism, one and only one nearly free vector bundle with exponents $(a,b)$ whose pencil of jumping lines has $P$ as base point. Moreover, we can think the matrix $M$, defining the nearly free vector bundles in Definition \ref{nf-def}, as
$$
{}^tM = [x,y,z^{b-a+1}].
$$
\end{Proposition}
\begin{proof}
We have that, due to Theorem \ref{first-thm}, $\TT_C$ corresponds to an element of $$\Ext^1(\II_P(-b+1),\OO_{\PP^2}(-a))\simeq H^1(\II_P(-b+a-2))^\lor.$$ Using Serre duality, and having supposed that $b\geq a$, we have that $h^1(\II_P(-b+a-2))=1$, therefore we have a unique non trivial extension.

Moreover, by a change of coordinates we can choose a simple presentation of $\TT_C$. Indeed, given the short exact sequence in (\ref{NFx}), we can apply a change of coordinates such that the point $P$ is defined by $\{x=y=0\}$. This means that the matrix $M$ defining the bundle $\TT_C$ is of the form
$$
M = [x, \: y, \: z^{b-a+1} + h]^t
$$
with $h = xh_0 + yh_1$, where $h_0$ and $h_1$ are homogeneous polynomials of degree $b-a$ in the coordinates $x,y,z$. Indeed, the third column of the matrix is a homogeneous form $f$ of degree $b-a+1$ in the variables $x,y,z$. Because of the fact that $M$ defines a vector bundle, we get that $z^{b-a+1}$ must be a summand of $f$, or else $M$ would vanish when evaluated at the point $(0:0:1)$. The part $h$ describes the other summands, which necessaly have a $x$ or a $y$ in each summand; notice that $h$ can be zero. This implies that we have the following commutative diagram
$$
\xymatrix{
0 \ar[r] & \OO_{\PP^2}(-b-1) \ar[d]^{1} \ar[rr]^(0.4){M} & & \OO_{\PP^2}(-a) \oplus \OO_{\PP^2}(-b)^2 \ar[r] \ar[d]^{A} & \TT_C \ar[r]\ar[d]^{\simeq} & 0\\
0 \ar[r] & \OO_{\PP^2}(-b-1)  \ar[rr]^(0.4){[x,\: y,\: z^{b-a+1}]^t} & & \OO_{\PP^2}(-a) \oplus \OO_{\PP^2}(-b)^2 \ar[r]  & \tilde{\TT}_C \ar[r] & 0
}
$$
with $A = \left[
\begin{array}{ccc}
1 & 0 & 0 \\
0 & 1 & 0 \\
-h_0 & -h_1 & 1
\end{array}
\right]$.
\end{proof}

Let us now prove how the special configuration of the jumping lines observed before actually characterizes a nearly free vector bundle, in the non stable case.
\begin{Theorem}\label{characterization-by-jl}
Let $P\in \PP^2$ be a point  and $E$ be a rank $2$ vector bundle on $\PP^2$ that we can assume to be normalized, $c_1(E)=0$ or $c_1(E)=-1$. Assume that $S(E)=\{L, L\ni P\}$ and that $o(L)=1$ for any $L\in S(E)$. Then, we have the following options
\begin{itemize}
\item if $E$ is either unstable or semistable then  $E$ is a nearly free vector bundle;
\item if $E$ is stable, then $c_1(E)=-1$, $c_2(E)=4$ and it is defined by the resolution
$$
0 \longrightarrow \OO_{\PP^2}(-4) \stackrel{A}{\longrightarrow} \OO_{\PP^2}(-1) \oplus \OO_{\PP^2}(-2)^2\longrightarrow E \longrightarrow 0,
$$
where, after choosing $P = (0:0:1)$, $A=(f(x,y,z),x^2,y^2))$ for any cubic form $f$ not passing through $P$.
\end{itemize}
\end{Theorem}
\begin{proof}
Let us consider  the integer $a \in \ZZ$ such that $h^0(E(a)) \neq 0$ and $h^0(E(a-1))=0$. It implies that we have a short exact sequence 
$$
0 \longrightarrow \OO_{\PP^2} \longrightarrow E(a) \longrightarrow \II_Z(2a+c_1(E)) \longrightarrow 0,
$$
where  $Z\subset \PP^2$ is a 0-dimensional scheme of length $c_2(E(a))$.

$\bullet$ If $E$ is not stable, i.e. $a\leq0$, then for every line $L$ such that $L\cap Z = \emptyset$, we have that $E_{|L} \simeq \OO_L(-a) \oplus \OO_L(a+c_1(E))$ which gives $\delta(E)=|2a-c_1(E)|$. Indeed $E_{|L}$ corresponds to an element of $\Ext^1(\OO_L(a+c_1(E)), \OO_L(-a)) \simeq H^1(\OO_L(-2a-c_1(E)))=0$. Moreover, each line $L$ such that $L \cap Z \neq \emptyset$ is a jumping line for the bundle $E$ since the surjective map $E_{|L} \rightarrow \OO_L(a-1+c_1(E))$ induced by the restriction to $L$ gives $\delta_L(E)>\delta(E)$. Therefore, since we have assumed that any jumping line $L$ contains the point $P$ and has $o(L)=1$, $Z$ is the simple point $P \in \PP^2$, or else we would have at least one line with a bigger jump. This forces $E$ to be nearly free.

$\bullet$ Let us suppose now that $E$ is stable and consider the following diagram
$$
\xymatrix{
\tilde{\PP}^2 \ar[d]_p \ar[r]^q & \check{P}\\
\PP^2
}
$$
where $\tilde{\PP}^2$ denotes the blow-up of the projective plane along $P$ and $\check{P}$ the pencil of all lines passing through $P$. Recall that, from our assumptions, we have that $E_{|H} \simeq \OO_H\oplus \OO_H(c_1(E))$ for the generic line $H$ and for a jumping line $L$, i.e.  passing through $P$, we have $E_{|L} \simeq \OO_L(1)+\OO_L(-1+c_1(E))$. In particular this means that, for any line $L\in \check{P}$, we have $h^0(E_{|L}(-1))=1$, which implies that $q_*p^*E(-1)$ is an invertible sheaf on $\check{P}$, which we will denote by $\OO_{\check{P}}(-n)$. Therefore, we have a nonzero morphism 
$$
q^*\OO_{\check{P}}(-n) \stackrel{\phi\neq 0}{\longrightarrow} p^*E(-1).
$$
This morphism does not vanish at any point; indeed, if it does, it would also vanish at some point in at least one fiber $q^{-1}(L)$. On the other hand, we know that its restriction to $q^{-1}(L)$ is of the form (see \cite[page 53]{OSS})
$$
\phi_{q^{-1}(L)} : \OO_L \longrightarrow E_L(-1).
$$
Since, by assumption, for any $L\ni P$, we have $E_{|L}(-1) \simeq \OO_L \oplus \OO_L(-2 + c_1(E))$, the morphism $\phi_{q^{-1}(L)}$ cannot vanish.\\ 
Then, we obtain the following exact sequence
$$
0 \longrightarrow q^*\OO_{\check{P}}(-n) \longrightarrow p^*E(-1) \longrightarrow p^*\OO_{\PP^2}(-2 + c_1(E)) \otimes q^*\OO_{\check{P}}(n) \longrightarrow 0
$$
and hence
$$
0 \longrightarrow q^*\OO_{\check{P}} \longrightarrow p^*E(-1)\otimes q^*\OO_{\check{P}}(n) \longrightarrow p^*\OO_{\PP^2}(-2 + c_1(E)) \otimes q^*\OO_{\check{P}}(2n) \longrightarrow 0.
$$
Using the projection formula and recalling that $p_*q^*\OO_{\check{P}}(n) \simeq \II_P^n(n)$ (see \cite[Section 4]{FV}),  we obtain the following commutative diagram
$$
\xymatrix{
0 \ar[r] & \OO_{\PP^2} \ar[d]_{\simeq} \ar[r] & E \otimes \II_P^n(n-1) \ar[d] \ar[r] &
\II_P^n(2n-2+c_1(E)) \ar[d] \ar[r] &0\\
0 \ar[r] & \OO_{\PP^2} \ar[r] & E(n-1) \ar[r] & \II_{\Gamma}(2n-2+c_1(E)) \ar[r] & 0
}
$$
which forces $n\ge 2$ because, having supposed $E$ to be stable, we have that $h^0(E(t))=0$ for $t \leq 0$ (see again \cite[page 132]{B}). The diagram is constructed considering the composition of the injective sheaf maps 
$$
\OO_{\PP^2} \longrightarrow E \otimes \II_P^n(n-1)
$$
and
$$
\II_P^n(n-1) \longrightarrow \OO_{\PP^2}(n-1)
$$
and we denote by $\Gamma$ the 0-dimensional scheme given by the zero locus of the section $s$ defined by the composition.\\
It means that $\Gamma$ contains the fat point of multiplicity $n$ defined by $\II_P^n$ and that the set theoretic support of $\Gamma$ is given by the point $P$.
This implies that $\Gamma$ is a local complete intersection supported on $P$, and therefore a global complete intersection. Moreover, because of the mentioned properties, the ideal defining $\Gamma$ is $(g,h)$, with $g$ and $h$ two homogeneous $n$-forms, each one product of $n$ linear forms, all representing lines passing through $P$.\\
As a consequence, we have the following two short exact sequences
$$
\xymatrix{
& &  & 0 \ar[d] &\\
& &  & \OO_{\PP^2} \ar[d]^s &\\
 &  &  & E(n-1) \ar[d] \\
0 \ar[r] & \OO_{\PP^2}(-2+c_1(E))\ar[r]  & \OO_{\PP^2}(n-2+c_1(E))^2  \ar[r]^{\alpha} & \II_\Gamma(2n-2 +c_1(E)) \ar[r] \ar[d]  & 0\\
& & & 0
}
$$
After that, we apply the functor $\Hom\left(\OO_{\PP^2}(n-2+c_1(E))^2,- \right)$ to the right column, and 
being $\Ext^1\left(\OO_{\PP^2}(n-2+c_1(E))^2,\OO_{\PP^2}\right)=0$, we can lift the morphism $\alpha$ to an element $\tilde{\alpha} \in \Hom\left(\OO_{\PP^2}(n-2+c_1(E))^2,E(n-1)\right)$. Combining it with the section $s\neq 0$, we obtain
$$
\xymatrix{
& & 0 \ar[d] & 0 \ar[d] &\\
& & \OO_{\PP^2} \ar[d]  \ar[r]^{\lambda} & \OO_{\PP^2} \ar[d] &\\
 &  & \OO_{\PP^2} \oplus \OO_{\PP^2}(n-2+c_1(E))^2\ar[d] \ar[r]^(0.65){(s,\tilde{\alpha})} & E(n-1) \ar[d] \\
0 \ar[r] & \OO_{\PP^2}(-2+c_1(E)) \ar[r] & \OO_{\PP^2}(n-2+c_1(E))^2  \ar[r] \ar[d]& \II_\Gamma(2n-2 +c_1(E)) \ar[r] \ar[d] & 0\\
& & 0 & 0 &
}
$$
with $\lambda\neq0$, which leads to the following commutative diagram
$$
\xymatrix{
& & 0 \ar[d] & 0 \ar[d] &\\
& & \OO_{\PP^2} \ar[d] \ar[r]^\simeq & \OO_{\PP^2} \ar[d] &\\
0 \ar[r] & \OO_{\PP^2}(-2+c_1(E)) \ar[r]^(0.4){[f,\: h,\:-g]^t} \ar[d]_{\simeq} & \OO_{\PP^2} \oplus \OO_{\PP^2}(n-2+c_1(E))^2\ar[d] \ar[r] & E(n-1) \ar[d] \ar[r] & 0\\
0 \ar[r] & \OO_{\PP^2}(-2+c_1(E)) \ar[r]^(0.5){[h,\: -g]^t} & \OO_{\PP^2}(n-2+c_1(E))^2  \ar[r] \ar[d]& \II_\Gamma(2n-2 +c_1(E)) \ar[r] \ar[d] & 0\\
& & 0 & 0 &
}
$$
Directly from the previous diagram, consider an element of $\Hom\left(\OO_{\PP^2}(n-2+c_1(E)),E(n-1)\right)$ and notice that $\Ext^1\left(\OO_{\PP^2}(n-2+c_1(E)),\OO_{\PP^2}(-2+c_1(E))\right)=0$, then it is possible to induce the following one
$$
\xymatrix{
& & 0 \ar[d] & 0 \ar[d]\\
& & \OO_{\PP^2}(n-2+c_1(E)) \ar[d] \ar[r]^\simeq & \OO_{\PP^2}(n-2+c_1(E)) \ar[d]\\
0 \ar[r] & \OO_{\PP^2}(-2+c_1(E)) \ar[r]^(0.4){[f,\: h,\:-g]^t} \ar[d]_{\simeq} & \OO_{\PP^2} \oplus \OO_{\PP^2}((n-2+c_1(E)))^2 \ar[d] \ar[r] & E(n-1) \ar[d] \ar[r] & 0\\
0 \ar[r] & \OO_{\PP^2}(-2+c_1(E)) \ar[r]^(0.4){[f,\:h]^t} & \OO_{\PP^2} \oplus \OO_{\PP^2}((n-2+c_1(E))) \ar[r] \ar[d]& \II_\Lambda(n) \ar[r] \ar[d] & 0\\
& & 0 & 0
}
$$
where $\Lambda:= \{f=h=0\}$ is a non empty 0-dimensional scheme of length $2n$ when $c_1(E)=0$ or $3n$ when $c_1(E)=-1$. Moreover, since the matrix $[f,\: h,\:-g]^t$ defines a vector bundle, $P\notin \mathrm{supp}(\Lambda)$. Tensoring the right column by $\OO_{\PP^2}(1-n)$ we get the short exact sequence
\begin{equation}\label{exact-lambda}
0 \longrightarrow \OO_{\PP^2}(-1+c_1(E)) \longrightarrow E \longrightarrow \II_\Lambda(1) \longrightarrow 0.
\end{equation}
Let us divide the study in two cases, $c_1(E)=0$ and $c_1(E)=-1$. 

Let us consider first $c_1(E)=0$.\\
Notice that $\Lambda$ is a complete intersection of a conic $f=0$ that does not contain $P$ and a curve $h=0$ consisting of $n\ge 2$ lines through $P$; moreover, we can suppose that the curve defined by $h$ contains at least two distinct lines (this can be done substituting $h$ with a linear combination of the two previous forms $g$ and $h$). Hence there exists a line $L$ such that $P\notin L$ and $|L\cap \Lambda|\ge 2$. 
Any line $L$ with $|L\cap \Lambda|\ge 2$ is necessarily a jumping line and this leads to contradiction because we would have found a jumping line not passing through $P$.

\smallskip

Let us now consider the case $c_1(E)=-1$.\\ 
In this case the bundle $E$ is defined by the following short exact sequence
\begin{equation}\label{eq-jump}
0 \longrightarrow \OO_{\PP^2}(-2-n) \longrightarrow \OO_{\PP^2}(1-n) \oplus \OO_{\PP^2}(-2)^2\longrightarrow E \longrightarrow 0.
\end{equation}
Our goal is to prove that the only possible case is $n=2$.\\
For $n=2$, taking $P=(0:0:1)$ one can assume that $(g,h)=(g(x,y),h(x,y))$ and more particularly $(g,h)=(x^2,y^2)$. Indeed, a pencil of conics in $\PP^1$ with no common factor always contains two squares.
After dualizing the sequence (\ref{eq-jump}) and tensoring by $\OO_{\PP^2}(-2)$ we obtain
$$
0 \longrightarrow E(-1) \longrightarrow \OO_{\PP^2}(-1) \oplus \OO_{\PP^2}^2\longrightarrow \OO_{\PP^2}(2)  \longrightarrow 0.
$$
A line $L$ is a jumping line for $E$ if and only if  $H^0(E_{|L}(-1))\neq 0$, which is equivalent for the map
$$
 H^0(\OO_{L}^2)\longrightarrow H^0(\OO_{L}(2)) 
$$
not to be injective.\\ 
If $L$ is a line not passing through $P$, i.e. defined by an equation of type $z = ax +by$, then the restrictions to $L$ of the quadratic forms $x^2$ and $y^2$ are linearly independent, which implies that the line $L$ is not a jumping line.\\
On the other hand, it is immediate to see that the restrictions of the two quadratic forms on any line passing through $P$ are linearly dependent, and therefore the line $L$ is a jumping line. 
\smallskip

Let us now finish the proof, showing that, if $n\geq 3$, we always have a jumping line that does not contain $P$, contradicting our assumption on $S(E)$.\\ 
After dualizing the sequence (\ref{eq-jump}) and tensoring by $\OO_{\PP^2}(-2)$ we obtain
$$
0 \longrightarrow E(-1) \longrightarrow \OO_{\PP^2}(n-3) \oplus \OO_{\PP^2}^2\longrightarrow \OO_{\PP^2}(n)  \longrightarrow 0.
$$
A line $L$ is a jumping line for $E$ if and only if  $H^0(E_{|L}(-1))\neq 0$, in other words if and only if the following map, between two vector spaces with dimension respectively $n$ and $n+1$,
$$
 H^0(\OO_{L}(n-3) \oplus \OO_{L}^2)\stackrel{M}\longrightarrow H^0(\OO_{L}(n)) 
$$
is not injective.\\ 
Let $g(x,y)=\sum_i \alpha_i x^{n-i}y^{i}$,  $h(x,y)=\sum_i \beta_i x^{n-i}y^{i}$ be two $n$ forms defined by $n$ lines through $P=(0:0:1)$  and $f(x,y,z)$ be a cubic form such that $f(0:0:1)\neq0$, i.e. $f=z^3+\cdots$. Since we are looking for a line $L$ that does not pass through $P=(0:0:1)$, we can assume that its equation is given by $z=ax+by$.\\ 
Substituting the equation of the line in the cubic form we have $f(x,y,ax+by)=\sum_{0\le i\le 3} \gamma_i(a,b)x^{3-i}y^{i}$ where $ \gamma_i(a,b)$ are degree 3 polynomials (non homogeneous). Notice that we must have $\gamma_0(a,b)=a^3+\cdots$ and $\gamma_3(a,b)=b^3+\cdots$. Therefore, the matrix $M$ has the following form: 
$$M = \left[
\begin{array}{cccccc}
\alpha_0 & \beta_0 & \gamma_0(a,b) &              &           & \\
\alpha_1 & \beta_1 & \gamma_1(a,b) & \gamma_0(a,b)&            &\\
\alpha_2 & \beta_2 & \gamma_2(a,b) &\gamma_1(a,b) &  \ddots    &\\
\alpha_3 & \beta_3 & \gamma_3(a,b) & \gamma_2(a,b)&   \ddots   &\gamma_0(a,b)\\
\vdots & \vdots    &            & \gamma_3(a,b)&   \ddots   &\gamma_1(a,b)\\ 
\vdots & \vdots    &              &  & \ddots   &\gamma_2(a,b)\\ 
\alpha_n & \beta_n &             &           &     & \gamma_3(a,b)
\end{array}
\right].$$
Since the $g$ and $h$ vanish simultaneously only at the point $P$, we have necessarly that $\alpha_0 \cdot \beta_n \neq 0$ or $\alpha_n \cdot \beta_0 \neq 0$ and, without loss of generality we can assum to be in the first case. 
By linear combination of lines and the first two columns, we see that this matrix is equivalent to 
$ \left[
\begin{array}{cc}
I_2 & *\\
0 & N \\
\end{array}
\right]$, where $I_2=\left[
\begin{array}{cc}
1 & 0\\
0 & 1 \\
\end{array}
\right]$ and $N$ is a $(n-1)\times(n-2)$ matrix of degree $3$ polynomials in $(a,b)$. Consider $N_h = N_h(a,b,c)$ the ``homogenization'' of $N$ in order to have a matrix of cubic forms. Thom-Porteous formula tells us that the scheme where the rank of $N$ is less than $n-2$ (which implies that the rank of $M$ is less than $n$) is a scheme of length 
$\frac{(n-2)(n-1)}{2}3^2$ when finite, or it contains a curve.\\ 
We will now show that in either case we have a solution of type $(a:b:1)$, which corresponds to a jumping line not passing through $P$.\\
If it is finite, then it cannot have more than $3n-2$ points at the infinity line, given by the degree of each maximal minor of $N_h(a,b,0)$ (which corresponds to the restriction at the infinity line). Notice that 
$$
3n-2 < \frac{(n-2)(n-1)}{2}3^2 \:\:\mbox{ when } \:\: n>2,
$$
which shows that we must have at least a point in the affine plane where the rank of $N$, and therefore $M$, drops and this corresponds to a jumping line not passing through $P$.\\
If it contains a curve, with no solutions of type $(a:b:1)$, then all points of type $(a:b:0)$ are a solution.\\
In this case, 
we get that $\rk N_h(a,b,0)<n-2$ for all $(a,b)\neq (0,0)$, which implies that 
$$
\rk \left[
\begin{array}{cc}
I_2 & *_h(a,b,0)\\
0 & N_h(a,b,0)\\
\end{array}
\right]< n.
$$
Applying the inverse transformation that sent $M$ to the upper triangular block form, we get that
$$
\rk \left[
\begin{array}{cccccc}
\alpha_0 & \beta_0 & a^3 &              &           & \\
\alpha_1 & \beta_1 & 3a^2b & a^3&            &\\
\alpha_2 & \beta_2 & 3ab^2 &3a^2b &  \ddots    &\\
\alpha_3 & \beta_3 & b^3 & 3ab^2&   \ddots   &a^3\\
\vdots & \vdots    &            & b^3&   \ddots   &3a^2b\\ 
\vdots & \vdots    &              &  & \ddots   &3ab^2\\ 
\alpha_n & \beta_n &             &           &     & b^3
\end{array}
\right]< n.
$$
Notice that such matrix is exactly the matrix $M$ when we choose $f=z^3$. This would imply that all lines of $\PP^2$ are jumping lines for the vector bundle defined by
$$
0 \longrightarrow \OO_{\PP^2}(-2-n) \stackrel{[g,h,z^3]}{\longrightarrow} \OO_{\PP^2}(1-n) \oplus \OO_{\PP^2}(-2)^2\longrightarrow E \longrightarrow 0
$$
This gives a contradiction, because a vector bundle $E$ defined in this way is stable and, according to the Grauert-M\"ulich theorem (see \cite[page 206]{OSS}), its splitting type is balanced on the generic line.


\end{proof}

\begin{Corollary}
Consider a family of 0-dimensional complete intersection schemes $\Lambda = (\lambda g + \mu h, f)$, with $g=g(x,y)$, $h = h(x,y)$ of degree $n$ with no common factor, and a cubic form $f$ such that $f(0,0,1) \neq 0$, parametrized by $(\lambda:\mu)\in \PP^1$.\\
Then, it exists an element $(\lambda_0 :\mu_0)\in \PP^1$ such that the complete intersection $\Lambda_0 = (\lambda_0 g + \mu_0 h, f)$ has at least one 3-secant line not passing through $(0:0:1)$.
\end{Corollary}
\begin{proof}
The existence of a $\Lambda_0$ with a 3-secant line, not passing through $P$, is equivalent to have a jumping line, again not passing through $P$, for the vector bundle defined by
$$
0 \longrightarrow \OO_{\PP^2}(-2-n) \stackrel{A}{\longrightarrow} \OO_{\PP^2}(1-n) \oplus \OO_{\PP^2}(-2)^2\longrightarrow E \longrightarrow 0
$$
with $A = [f,g,h]$. Indeed, considering the sequence \eqref{exact-lambda}, relating $\Lambda_0$ and $E$, a 3-secant line $L$ to $\Lambda_0$ is clearly a jumping line for $E$.\\ 
Conversely, let $L$ be jumping line not passing through $P$. Since there is a pencil $\Lambda$ of complete intersections, $L$ meets at least one element $\Lambda_0$ of the pencil. Considering again the exact sequence \eqref{exact-lambda} defining $\Lambda_0$, and its restriction to $L$, we must have $|L\cap \Lambda_0|=3$.

\end{proof}
\subsection{Restriction on one line} In many cases we can decide if an arrangement is free by computing the Chern classes of its logarithmic associated vector bundle and determining its splitting type on one line. This can be done thanks to \cite[Corollary 2.12]{EF}. Indeed their result, concerning vector bundles on any $\PP^n$, implies in particular that a rank two vector bundle $E$ over $\PP^2$ such that $c_1(E)=-a-b$ and $c_2(E)=ab$ is the free bundle $\OO_{\PP^2}(-a) \oplus \OO_{\PP^2}(-b)$ if and only if there exists  one line $L$ such that $E|_{L}=\OO_{L}(-a) \oplus \OO_{L}(-b) $. 

We give now a similar statement  for nearly free vector bundles. Thanks to this we can, knowing its splitting type on one line, determine if a vector bundle is nearly free or not.
\begin{Proposition} \label{splitting}
  Let $E$ be a rank-$2$ vector bundle on $\PP^2$ and assume
  $c_1(E)=-r$ for some $r\ge 0$ and $c_2(E)=1$.
  Then, the following are equivalent:
  \begin{enumerate}
  \item \label{splits} the bundle $E$ is nearly free with exponents $(0,r+1)$,
  \item \label{one line} there is a line $L$ of $\PP^2$ such that $E|_{L} \simeq \OO_L
    \oplus \OO_L(-r)$.
  \end{enumerate}
\end{Proposition}
\begin{proof}
  Condition \ref{splits} clearly implies \ref{one line}.
  It remains to show that \ref{one line} implies
  \ref{splits}.
  
  Let $t$ be the smallest integer such that $H^0(\PP^2,E(t)) \ne 0$. If $t<0$ it is clear that there is no line $L$
  such that $E|_{L} \simeq \OO_L
    \oplus \OO_L(-r)$. Then we have $t \ge 0$.
  Also, it is well-known (cf. \cite[Lemmas 1 and 2]{B}) that any non-zero global section $s$ of $E(t)$
  vanishes along a subscheme $W$ of $\PP^2$ of codimension $\ge 2$ and
  of length:
  \begin{equation}
    \label{c2}
  c_2(E(t))=t(t-r)+1 \ge 0.
  \end{equation}
  We have an exact sequence:
  \[
  0 \to \OO_{\PP^2} \stackrel{s}{\rightarrow} E(t) \to \II_W(2t-r) \to 0.
  \]
  So $t=0$ would imply $W = \{P\}$ is one point and for any value of $r$, $E$ is nearly free with exponents $(0,r+1)$. 
  
  Hence, we assume $t>0$. Since $ c_2(E(t))=t(t-r)+1 \ge 0$ it implies $t=r$ or $t>r$.
  Let us consider first the case $t=r>0$.  Since $c_2(E(r))=1$ the scheme $W$ is a single point $P$. Since  
  $H^0(\PP^2,\II_P(r-1)) =H^0(\PP^2,\, E(-1))=0$ this implies $r=1$ and $E$ is again the tangent bundle which is nearly free.
  Since $h^0(E(-1))=0$ we have $h^0(\PP^2,\II_W(2t-r-1))=\binom{2t-r+1}{2}-t(t-r)-1\le 0$ ; when $t>r$ this never occurs. 
  \end{proof}
\section{Addition and deletion}\label{sec-addel}
In this section we study the behaviour of an arrangement obtained by deleting or adding a line to a free arrangement. In particular, we characterize when the obtained arrangement is nearly free and in this case we describe its pencil of jumping lines. The description we present will recover some results of Section 5 in \cite{AD} (see in particular Theorems 5.7, 5.10 and 5.11).

In the following proposition we describe how to construct a nearly free vector bundle deleting a line, satisfying specific properties, from a free arrangement. This process is known as \emph{deletion}.

\begin{Proposition}\label{PropDel}
A rank two vector bundle $E$ is nearly free with exponents $(a,b)$ if and only if it can be constructed as an extension in $\Ext^1(\OO_L(-b),\OO_{\PP^2}(-a) \oplus \OO_{\PP^2}(-b))$ where $L$ is a line and $a,b$ are integers such that $0\le a\le b$.\\ 
Moreover, considering any vector bundle $E$ given by an element of $\Ext^1(\OO_L(-t),\OO_{\PP^2}(-a) \oplus \OO_{\PP^2}(-b))$ where $L$ is a line and $t,a,b$ are integers such that $0\le a\le b$, then $E$ is a nearly free vector bundle if and only if $t = b$, which forces its exponents to be $(a,b)$.\\
\end{Proposition}
\begin{proof}
Let us consider a nearly free vector bundle defined by the resolution (\ref{NFdef}). Therefore we can choose an injective map $\OO_{\PP^2}(-a)\oplus \OO_{\PP^2}(-b) \longrightarrow E$ which gives us the following commutative diagram
\begin{equation}\label{diagDel}
\xymatrix{
& & 0 \ar[d] & \ar[d] 0 \\
& & \OO_{\PP^2}(-a) \oplus \OO_{\PP^2}(-b)\ar[d] \ar[r]^\simeq & \OO_{\PP^2}(-a)\oplus \OO_{\PP^2}(-b)\ar[d]\\
0 \ar[r] & \OO_{\PP^2}(-b-1) \ar[r] \ar[d]_\simeq & \OO_{\PP^2}(-a) \oplus \OO_{\PP^2}(-b)^2 \ar[d] \ar[r] & E \ar[d]\ar[r] & 0\\
0 \ar[r] & \OO_{\PP^2}(-b-1) \ar[r]  &  \OO_{\PP^2}(-b) \ar[d] \ar[r] &\OO_L(-b) \ar[d]\ar[r] & 0\\
& & 0 & 0
}
\end{equation}
Focus on the left column and let us discuss its geometrical meaning in the arrangement. Let $L\in C$ where $C$ is a free arrangement such that $\TT_{C}=\OO_{\PP^2}(-a)\oplus \OO_{\PP^2}(-b)$. According to \cite[Proposition 5.1]{FV},  we have a short exact sequence
$$
0 \longrightarrow \OO_{\PP^2}(-a) \oplus \OO_{\PP^2}(-b) \longrightarrow \TT_{C\backslash\{L\}} \longrightarrow \OO_L(-t) \longrightarrow 0
$$
where $t$ counts the number of triple points in the line $L$.
If we suppose $\TT_{C \backslash \{L\}}$ to be nearly free with exponents $(a,b)$, we get $t = b$ by computing the second Chern classes for instance. It shows that we can construct a nearly free arrangement with exponents $(a,b)$ by deleting a line in a free arrangement with the same exponents when this line passes through exactly $b$ triple points. That is why this process is known as \emph{deletion}.

To prove the second part of Proposition \ref{PropDel}, consider a vector bundle $E$ fitting in the following short exact sequence
$$
0 \longrightarrow \OO_{\PP^2}(-a) \oplus \OO_{\PP^2}(-b) \longrightarrow E \longrightarrow \OO_L(-t) \longrightarrow 0.
$$
Considering its dual exact sequence, we get a surjective map $\OO_{\PP^2}(a) \oplus \OO_{\PP^2}(b) \rightarrow \OO_L(t +1)$, which forces us to have either $t = a-1$ or $t\geq b-1$. If $t=a-1$ or $t = b-1$ then $E$ would be free (see for instance \cite[Proposition 5.2]{FV}), therefore we can only focus on $t \geq b$. If $t > b$ then $E$ cannot be nearly free since the surjective restriction map 
$$
E_{|L} \longrightarrow \OO_L(-t) \longrightarrow 0
$$
 implies that the splitting type on $L$ for $E$ has gap bigger than allowed by Proposition \ref{NFSplitting}.\\
Finally if $t = b$, we can recover Diagram (\ref{diagDel}), which implies that $E$ is nearly free by its resolution.
\end{proof}
In the second part of the section we describe the second operation on the arrangement, dual to the previous one, which is known as \emph{addition}. Similarly to the previous case, its geometrical interpretation corresponds to adding a line passing through a specific number of triple points of the original arrangement.

\begin{Proposition}\label{PropAdd}
A rank two vector bundle $E$ is nearly free with exponents $(a+1,b+1)$ if and only if it can be constructed as the kernel of surjective map in $\Hom(\OO_{\PP^2}(-a)\oplus\OO_{\PP^2}(-b), \OO_{L}(1-a))$ where $L$ is a line and $a,b$ are integers such that $0\le a\le b$.\\ 
Moreover, considering any kernel $E$ of the surjective map given by an element in $\Hom(\OO_{\PP^2}(-a)\oplus\OO_{\PP^2}(-b), \OO_{L}(-t))$ where $L$ is a line and $t,a,b$ are integers such that $0\le a\le b$, then $E$ is a nearly free vector bundle if and only if $t = a-1$, which forces its exponents to be $(a+1,b+1)$. 
\end{Proposition}
\begin{proof}
If $\TT_C$ is nearly free, we can consider the right column in Diagram (\ref{diagDel}) and, taking its dual, we obtain what required.\\
In order to prove the other implication, we take a vector bundle who is defined by the short exact sequence
$$
0 \longrightarrow E \longrightarrow \OO_{\PP^2}(-a) \oplus \OO_{\PP^2}(-b) \stackrel{f}{\longrightarrow} \OO_L(-t) \longrightarrow 0
$$
Notice that in order for the map $f$ to be surjective, we must have $t = b$ or $t \leq a$. Once again, taking its dual and applying Proposition \ref{PropDel}, we get that $E$ is nearly free if and only if $t = a-1$.\\
As before $t$ can be interpreted as the number of triple points through which the line added in the arrangement must pass.
\end{proof}
We end this section relating the jumping point of a nearly free vector bundle with the two operations described above.
\begin{Proposition}
\label{45}
Let $\TT_C$ be a nearly free vector bundle whose associated arrangement $C$ is constructed adding or deleting a line $L$ in a free arrangement. Then the jumping point of $\TT_C$ belongs to the line $L$.
\end{Proposition}
\begin{proof}
The result comes immediately from Diagram (\ref{diagDel}) for the deletion operation. For the addition, we consider the dual exact sequence of the one defining the bundle and again, we conclude using the same commutative diagram.
\end{proof}
\begin{Remark}
It is not always possible to add or to delete a line from a given free arrangement in order to find a nearly free arrangement. For instance it is not possible to find a nearly free arrangement by deleting a line from the Hesse arrangement ($12$ lines through the nine inflexion points of a smooth cubic curve) since there is no line containing $7$ triple points.

On the other hand, consider the following free arrangement with exponents $(4,4)$ consisting in two sets of four lines, the first one passing through the point $(1:0:0)$ and the second one through $(0:1:0)$ plus the infinity line, i.e. the line defined by the two previous points. Choose the  eight ``finite'' lines in order not to have three points, coming from the intersections, aligned. Then it is not possible to add a line to the previous arrangement that passes through three triple points, and therefore, by Proposition \ref{PropAdd}, it is not possible to obtain a nearly free arrangement starting from the free given one. 

In the following section we give a family of free arrangements from which we can always build a nearly free arrangement by deletion or addition.
\end{Remark}
\section{Nearly free arrangements obtained by addition and deletion from a free one}\label{sec-ex}
The previous sections give us necessary and sufficient conditions in order to construct a nearly free vector bundle with exponents $(a,b)$ starting from a free vector bundle and applying addition or deletion. Indeed, in this section we will show specific examples that realize such construction. Moreover, we will be able to determine which lines of the associated arrangement are jumping.

In order to determine the jumping order on lines of the given arrangement, we will use the multiarrangements introduced by Ziegler in \cite{Z}.

\smallskip

Let us recall some results about these multiarrangements on lines. Let $C$ be an arrangement of $N$ lines. Let $L\in C$. We denote by $n$ the number (without multiplicity) of intersection points on $L$ and $m_1\ge \cdots \ge m_n$ their multiplicities. We have of course $\sum_i m_i=N-1$. If there is  no triple point of $C$ on $L$ then $n=N-1$, if $L$ contains some triple points then $n< N-1$. Now, according to Case 2.1, Case 2.2 on page 3 and Theorem 3.1 in  \cite{WY},  we have:
\begin{itemize}
\item If $m_1\ge \sum_{i=2}^n m_i$ then the splitting type is $(\sum_{i=2}^n m_i, m_1)$.
\item If $2n-1\ge N$ then the splitting type on $L$ is $(N-n,n-1)$.
\item If $2n-1\le N$ then the splitting type is balanced when the $n$ intersection points are in general position but can be unbalanced for special positions.
\end{itemize}

Let us consider an arrangement $C_0$ of $a+b+1$ lines ($0\le a\le b$ as usual) consisting in one line at infinity, $b$ parallel lines, $a-1$ parallel lines in another direction and one isolated line containing $a-1$ triple points (see Figure \ref{fig5}). This arrangement is free with exponents $(a,b)$. Indeed the associated vector bundle has the Chern classes of 
$\OO_{\PP^2}(-a)\oplus \OO_{\PP^2}(-b)$ and the splitting type on a vertical line is $(a, b)$ since $n=b+1$.

\begin{figure}[ht!]
\begin{center}
\includegraphics[scale=0.44]{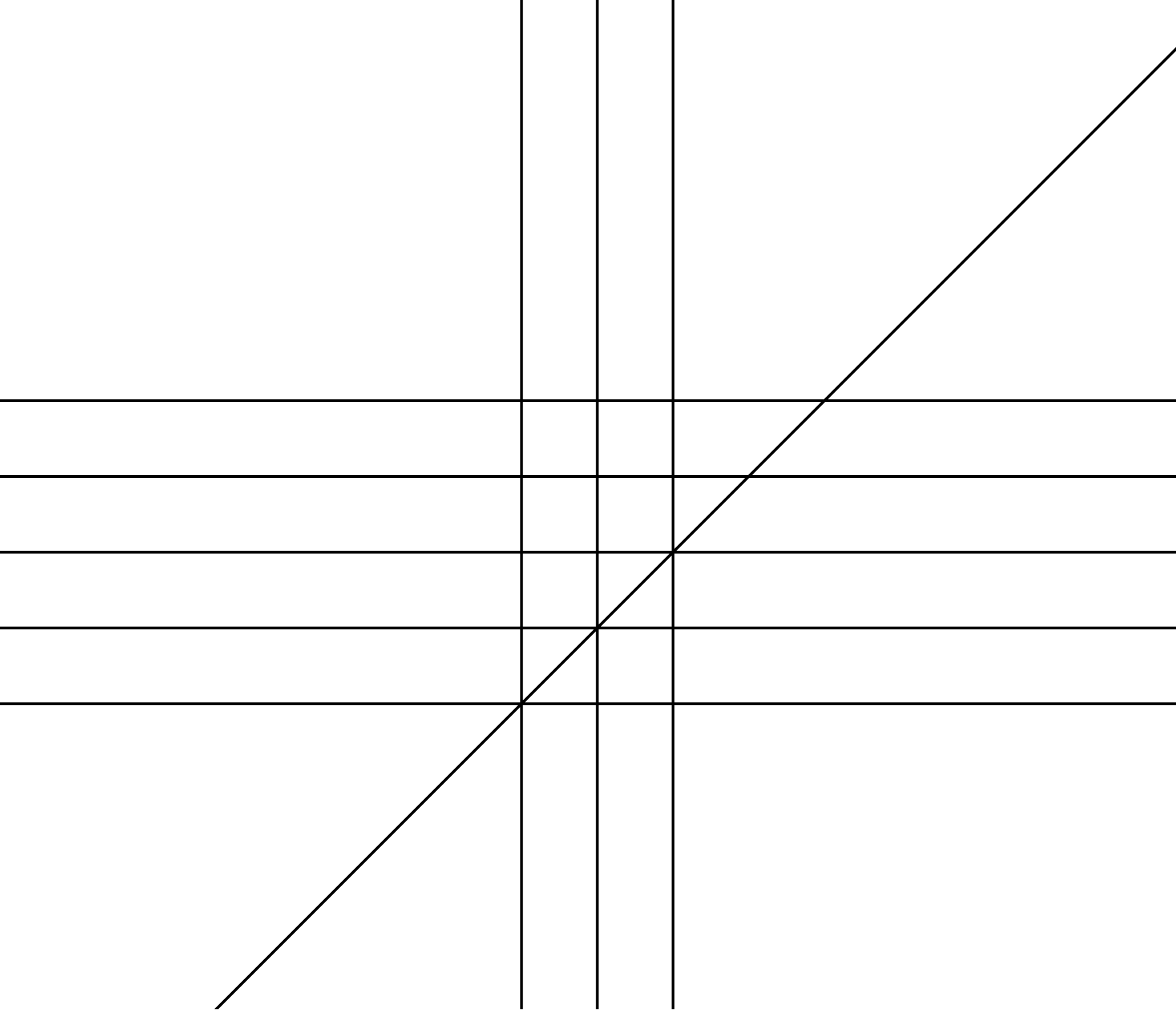}
\caption{}
\label{fig5}
\end{center}
\end{figure}

\subsection{Deletion}\label{eg-deletion}
A nearly free arrangement is obtained by deleting one line passing through $b$ triple points of the free arrangement $C_0$. The dotted line (see Figure \ref{fig6}) passes through $b-1$ triple points at infinity and one in the affine plane. Deleting this line we obtain a nearly free arrangement with exponents $(a,b)$ as it is proved in Proposition \ref{PropDel}.\\
The two red lines are the only lines of the arrangement that are jumping lines. Indeed the multiplicities of the multiarrangement on the diagonal line are $m_1=\cdots =m_{a-2}=2$ and 
$m_{a-1}=\cdots =m_{b+1}=1$. Then $n=b+1$ and since $2b+1> a+b$, the splitting type is $(a-1,b)$.\\
In the same way the multiplicities of the multiarrangement on the vertical red line are $m_1=a-1$ and $m_2=\cdots =m_{b+1}=1$. Then $n=b+1$ and  the splitting type is $(a-1,b)$.\\
On the contrary, since the multiplicities on a horizontal line that does not pass through $P$ and that does not contain a triple point out of infinity (it exists since $a\le b$) are $m_1=b-1$ and $m_2=\cdots =m_{a+1}=1$, its splitting type is $(a, b-1)$.\\
It shows that the generic splitting is $(a, b-1)$,  that the two lines through $P$ determine the jumping point, which will of course be $P$ itself, and that these two lines are the only jumping lines in $C$.

\medskip

\begin{figure}[ht!]
\begin{center}
\includegraphics[scale=0.44]{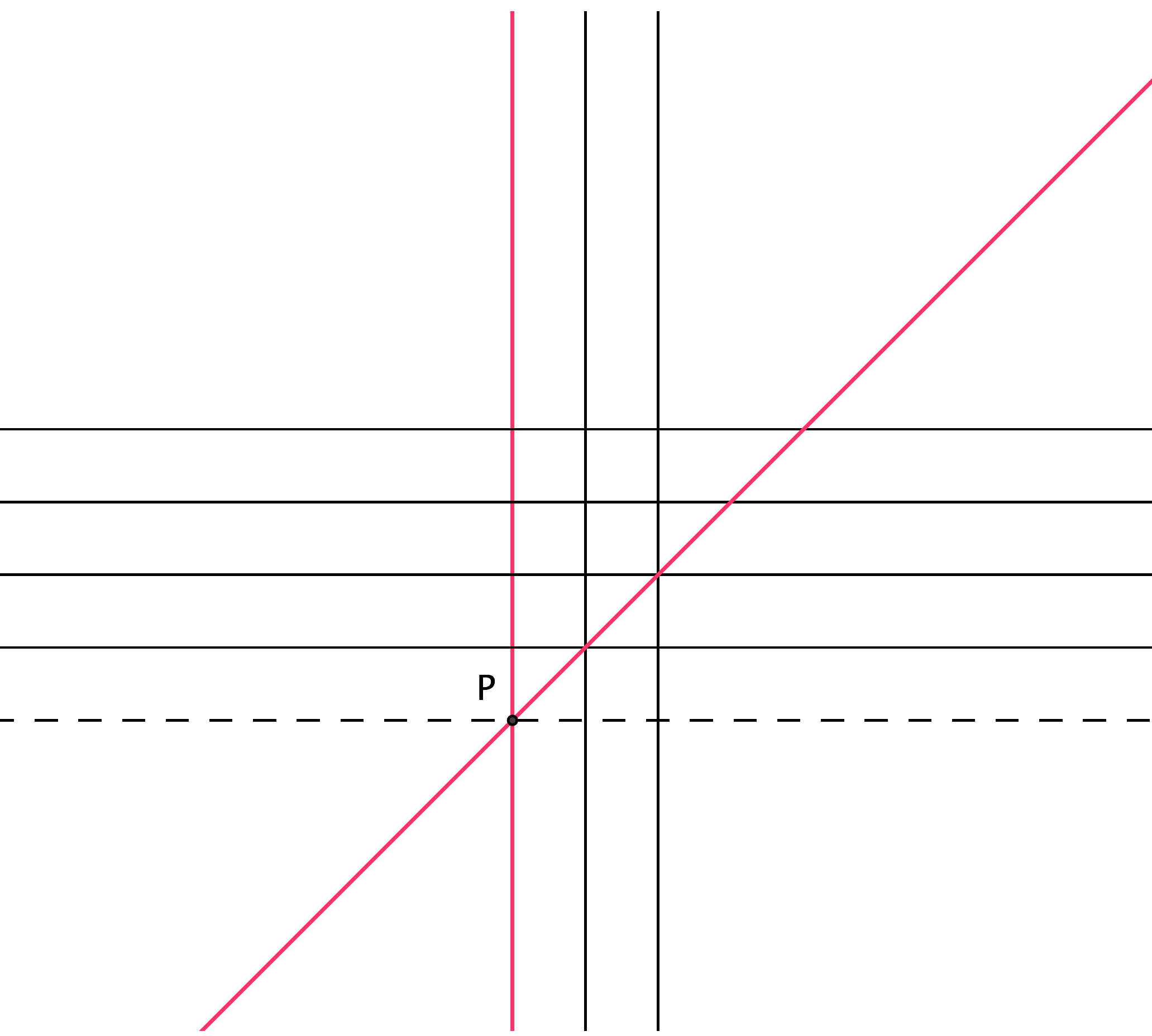}
\caption{}
\label{fig6}
\end{center}
\end{figure}

\subsection{Addition}\label{eg-addition}
We build now a nearly free vector bundle with exponents $(a+1,b+1)$ by adding one line (the blue line in Figure \ref{fig7}) passing through $a-1$ triple points to the free arrangement $C_0$ (see Proposition \ref{PropAdd}).

The point $P$ of intersection of the diagonal line of $C_0$ and the blue line is the jumping point of the nearly free arrangement. Indeed, $P$ belongs to the added line by Proposition \ref{45} and since 
the multiplicities on the diagonal line are $m_1=\cdots =m_{a-1}=2$ and $m_a=\cdots m_{b+2}=1$ the splitting type is $(a,b+1)$ which proves that $P$ belongs also to the diagonal line.

On the contrary, the splitting type on a horizontal line that does not contain nor $P$ neither a triple point is 
$(a+1,b)$ since the multiplicities are $m_1=b$ and $m_2=\cdots =m_{a+2}=1$. This proves that the generic splitting is this one and that the two lines of $C$ through $P$ are the only jumping lines of $C$.

\medskip

\begin{figure}[ht!]
\begin{center}
\includegraphics[scale=0.44]{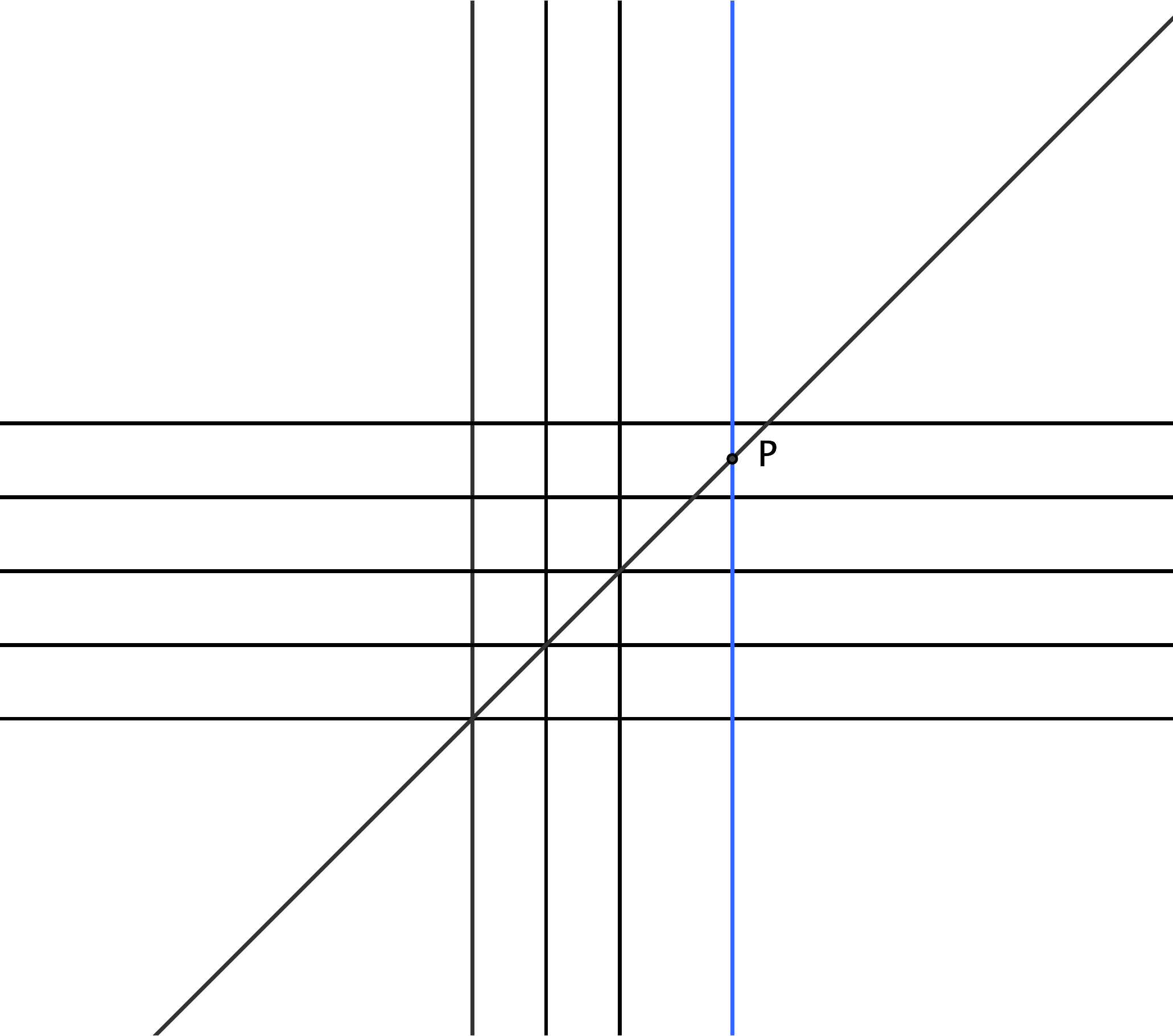}
\caption{}
\label{fig7}
\end{center}
\end{figure}

As a consequence of Proposition \ref{prop-detpoint}, we obtain the following result
\begin{Corollary}
Any nearly free vector bundle can be obtained from a free arrangement by addition or deletion.
\end{Corollary}
\begin{proof}
We have seen that, up to isomorphism, a nearly free vector bundle depends only on the exponent $(a,b)$ and the jumping point. Fixing them, it is always possible to construct with the two previous techniques a nearly free vector bundle with the given jumping point and exponent, this last depending on the number of lines.
\end{proof}

\section{Behaviour of the jumping point}\label{sec-jump}
Let us consider now a line arrangement $C$ that is nearly free. Then there exists an associated jumping point $P$. A natural question in the context of Terao's conjecture is the following one:
\begin{center}
\emph{Does this jumping point depend on the combinatorics of $C$?}
\end{center}
Actually the answer itself depends on the combinatorics we choose. Some combinatorics determine the position of the jumping point when some other combinatorics are not enough to determine its position. We give now some examples which tell us that we can have all possibilities for the position of $P$ relatively to $C$:
\begin{itemize}
\item the jumping point of $\TT_C$ is the intersection of at least two lines of the arrangement,
\item the jumping point is on one and only one line of the arrangement,
\item the jumping point does not belong to the arrangement, i.e. all the lines of the arrangement have generic splitting type.
\end{itemize}
We will show how the jumping points variates in the projective plane when shifting a line in order to maintain its combinatoric, see Example \ref{ex-1}. Moreover, we will show that for some fixed combinatorics, the jumping points will be forced to belong to the arrangement, see Example \ref{ex-line}, while for another fixed one, the jumping point can either belong or not to the arrangement, see Example \ref{ex-inout}.
\begin{Example}\label{ex-1}
Consider the arrangement of lines in $\PP^2$ defined by $$C:= x yz(x-z)(x+z)(y-z)(y+z)(x-y)(x+y)(x+ty-(1+t)z)=0$$ with $t\in \CC$ and represented in Figure \ref{fig1}.\\
\begin{figure}[ht!]
\begin{center}
\includegraphics[scale=0.10]{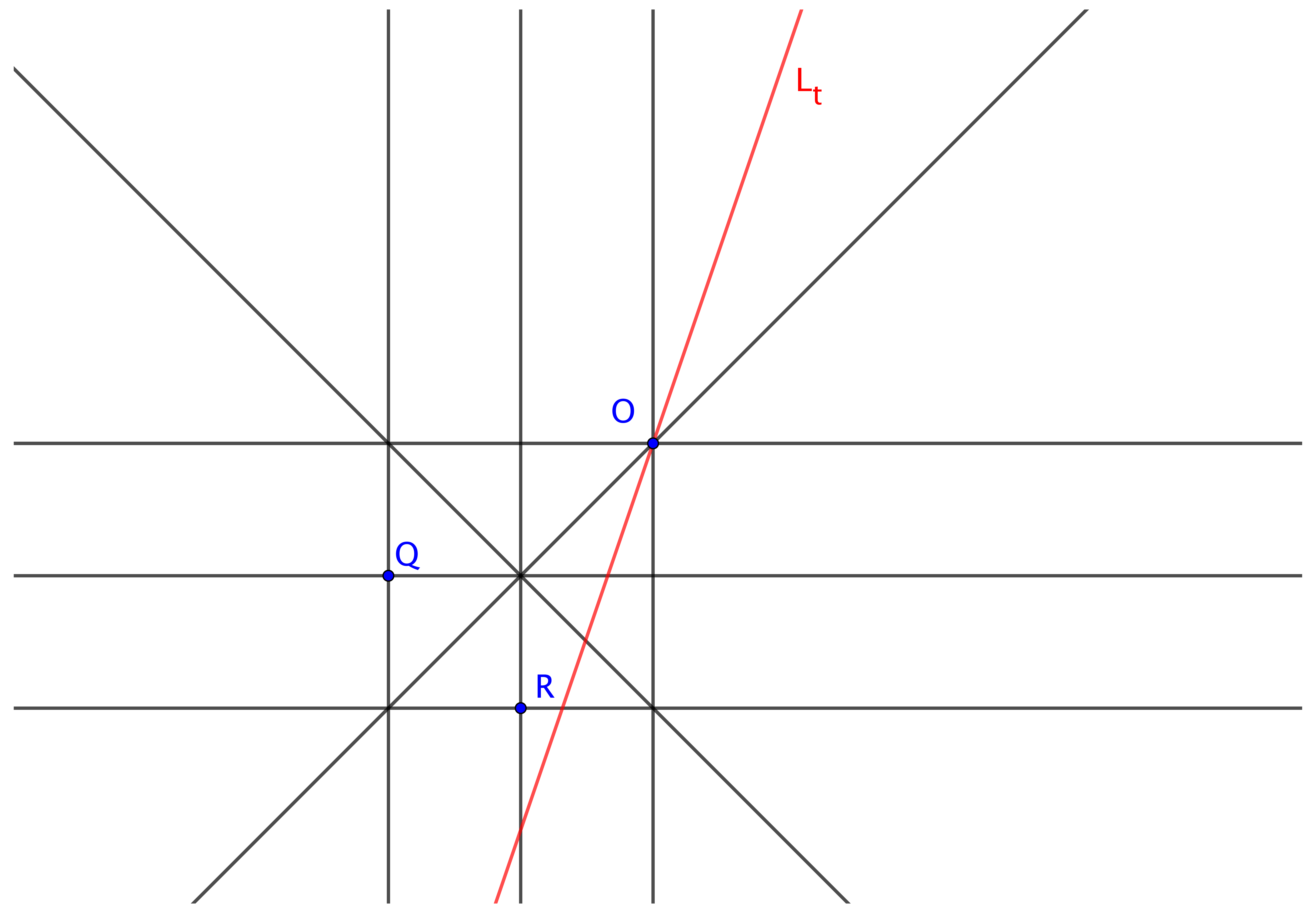}
\caption{}
\label{fig1}
\end{center}
\end{figure}
Notice that, in the figure, we omit the line at infinity, which nevertheless it is present in the arrangement, and the red line depends on the choice of the value $t$. It is possible to compute, for example using Macaulay2 (see \cite{GS}), that for any choice of $t$ except for the ones that give us the lines passing through $Q$ or $R$ or any line already present in the arrangement, the associated vector bundle is nearly free and the jumping point is given by the unique point in the projective plane which is the solution of the linear system
$$
\left\{
\begin{array}{l}
x+y+z=0\\
x+ty-(1+t)z=0.
\end{array}
\right.
$$
We first remark that this arrangement is obtained by adding the line $L_t$ to the free arrangement (called $B3$ in the literature; its exponents are $(3,5)$) consisting of the $9$ remaining lines.
As shown in section \ref{sec-addel}, the point always belongs to the line $L_t$. Therefore any such arrangement defining a nearly free vector bundle contains only one jumping line. If we consider the line passing either through the point $Q$ or $R$, we obtain a free arrangement. In conclusion this example gives a family of nearly free vector bundles
 which are parametrized by an open subset $\mathcal{U}$ of the line $x+y+z=0$, where each point of the open subset considered represents the jumping point of the associated nearly free vector bundle. Indeed, the arrangement will determine a nearly free vector bundle if and only if the line $L_t$ does not pass through any triple points except for $O$. This means that the cases we have to exclude are five: the line passing through $O$ and $Q$, the line passing through $O$ and $R$ and when $L_t$ coincides with a line already in the arrangement, i.e the lines $x-y=0$, $x-z=0$ and $y-z=0$. Therefore, $\mathcal{U} = \{x+y+z=0\} \setminus \{Q,R,(1:1:-2),(1:-1:1),(-1:1:1)\}$. 
\end{Example}
The following example will be constructed by adding one line to a free arrangement. We have already notice that this choice implies that the jumping point will belong to it. Nevertheless, in the first part of the example we will show that jumping point can either be on one or multiple lines of the arrangement. In the second part we will show that a well chosen combinatorics forces the jumping point to belong to one and only one line of the arrangement.
\begin{Example}\label{ex-line}
Consider the arrangement defined by $$C:=xyz(x^2-z^2)(y^2-z^2)(x-y)(x-y+2z)=0.$$ The associated vector bundle $\TT_C$ is nearly free (the line $x-y+2z=0$ is added to the free arrangement with exponents $(3,4)$ and contains two triple points) and it is given by the following resolution
$$
0 \longrightarrow \OO_{\PP^2}(-6) \stackrel{A}{\longrightarrow} \OO_{\PP^2}(-4) \oplus \OO_{\PP^2}(-5)^2 \longrightarrow \TT_C \longrightarrow 0
$$
with 
$$
A = \left[ \begin{array}{ccc} 9y^2 + 9yz, & -4x-5y+z, & 5x +13y - 8z\end{array}\right]^t
$$
and its jumping point is $P=(-1:1:1)$. In Figure \ref{fig2}, we can see the arrangement, from which we omit the infinity line. The red lines are the jumping lines in the arrangement. In this case the jumping point is an intersection of two jumping lines.\\
\begin{figure}[ht!]
\begin{center}
\includegraphics[scale=0.10]{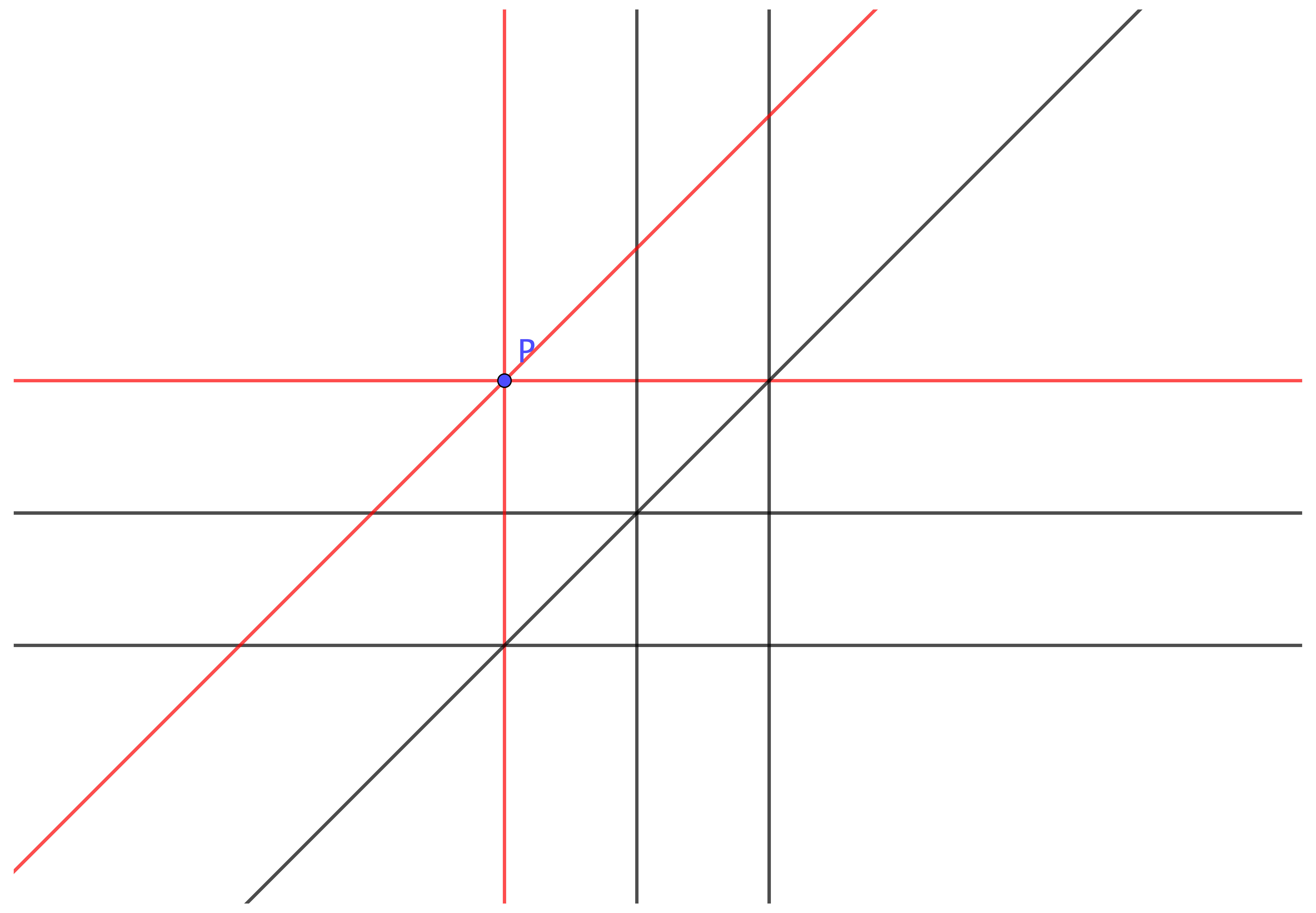}
\caption{}
\label{fig2}
\end{center}
\end{figure}
Let us consider now an arrangement with the exact same combinatorics of the previous one, but obtained for it by \emph{sliding} two perpendicular lines, i.e. $$C' :=(x+\frac{1}{2}z)(y+\frac{1}{2}z)z(x^2-z^2)(y^2-z^2)(x-y)(x-y+2z)=0.$$ The associated vector bundle $\TT_{C'}$ is also nearly free and it is given by the following resolution
$$
0 \longrightarrow \OO_{\PP^2}(-6) \stackrel{A}{\longrightarrow} \OO_{\PP^2}(-4) \oplus \OO_{\PP^2}(-5)^2 \longrightarrow \TT_C \longrightarrow 0
$$
with 
$$
A = \left[ \begin{array}{ccc} 18y^2 + 27yz + 9z^2, & 5x-13y-2z, & -8x -10y - 4z\end{array}\right]^t
$$
and its jumping point is $P'=(-4:2:3)$. Notice that jumping point has moved along the line $x-y+2z=0$ but it is no long intersection of two jumping lines of the arrangement. This situation is described in Figure \ref{fig3}. Notice that these two arrangements are constructed by adding a line ($x-y+2z=0$ in both cases) to a free arrangement.
\begin{figure}[ht!]
\begin{center}
\includegraphics[scale=0.10]{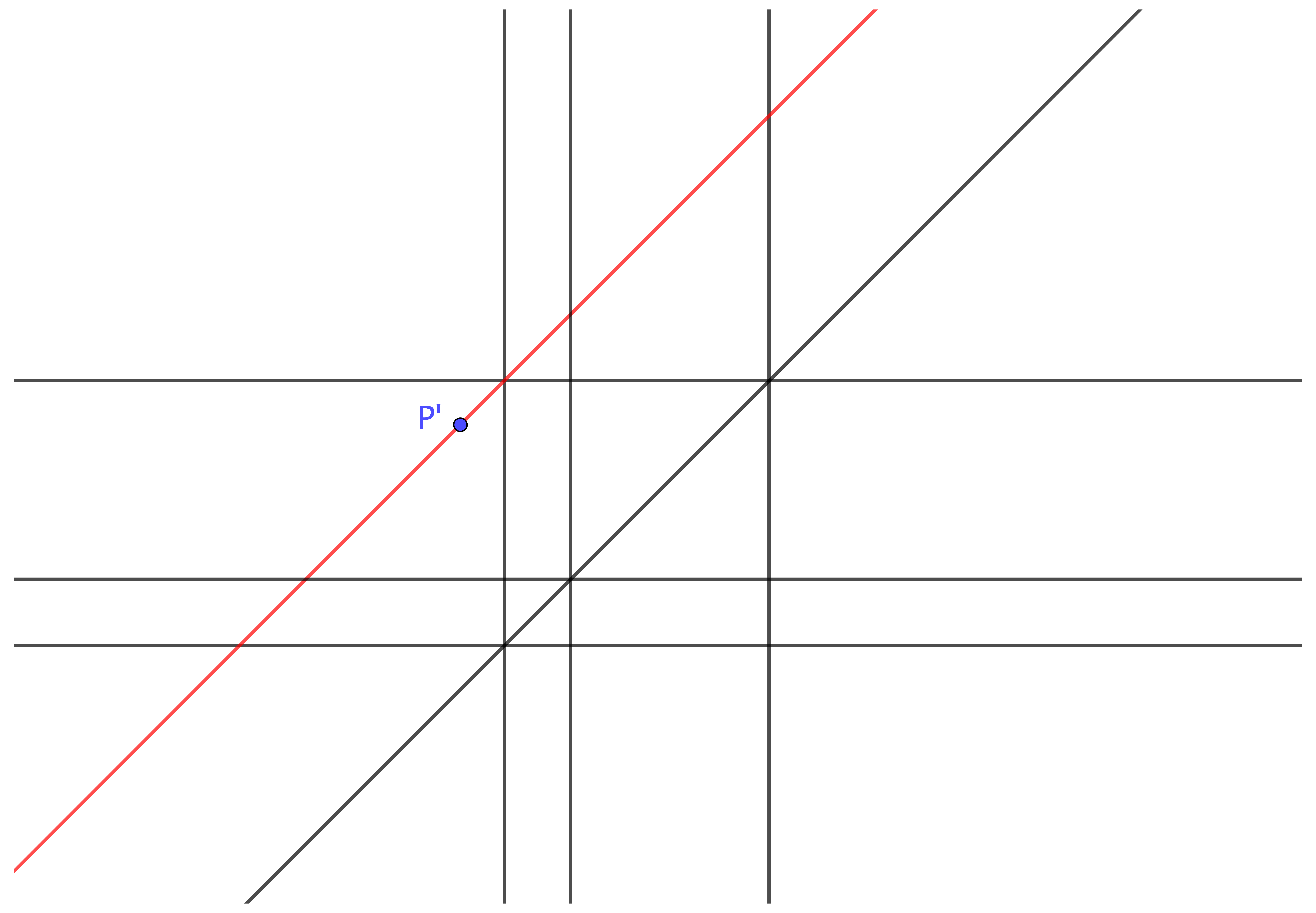}
\caption{}
\label{fig3}
\end{center}
\end{figure}

\smallskip

Generalizing the previous example taking more lines, the combinatorics of the arrangement can determine if the added line is the only jumping one.\\
Indeed, consider a free arrangement $\mathcal{A}$ with exponent $(4,5)$ given by two groups of four parallel lines, in the affine plane, having two different directions, one diagonal line passing through $4$ points of intersection of the grid formed by previous ones and the line at infinity. Obviously, we can choose  $z=0$ for the line at infinity, and the grid formed by vertical and horizontal lines, 
i.e. $x=\alpha_i z$ and $y=\beta_i z$ ($1\le i\le 4$). Then in order for the diagonal line to contain $4$ triple points, we must have $\alpha_i=\beta_i$.\\ 
Let us add a further line $D$, parallel to the diagonal one, passing through $2$ intersection points of the grid.\\
By Proposition \ref{PropAdd}, we obtain a nearly free arrangement $\mathcal{A}\cup D$ with exponents $(5,6)$, depicted in Figure \ref{fig9}.
\begin{figure}[ht!]
\begin{center}
\includegraphics[scale=0.10]{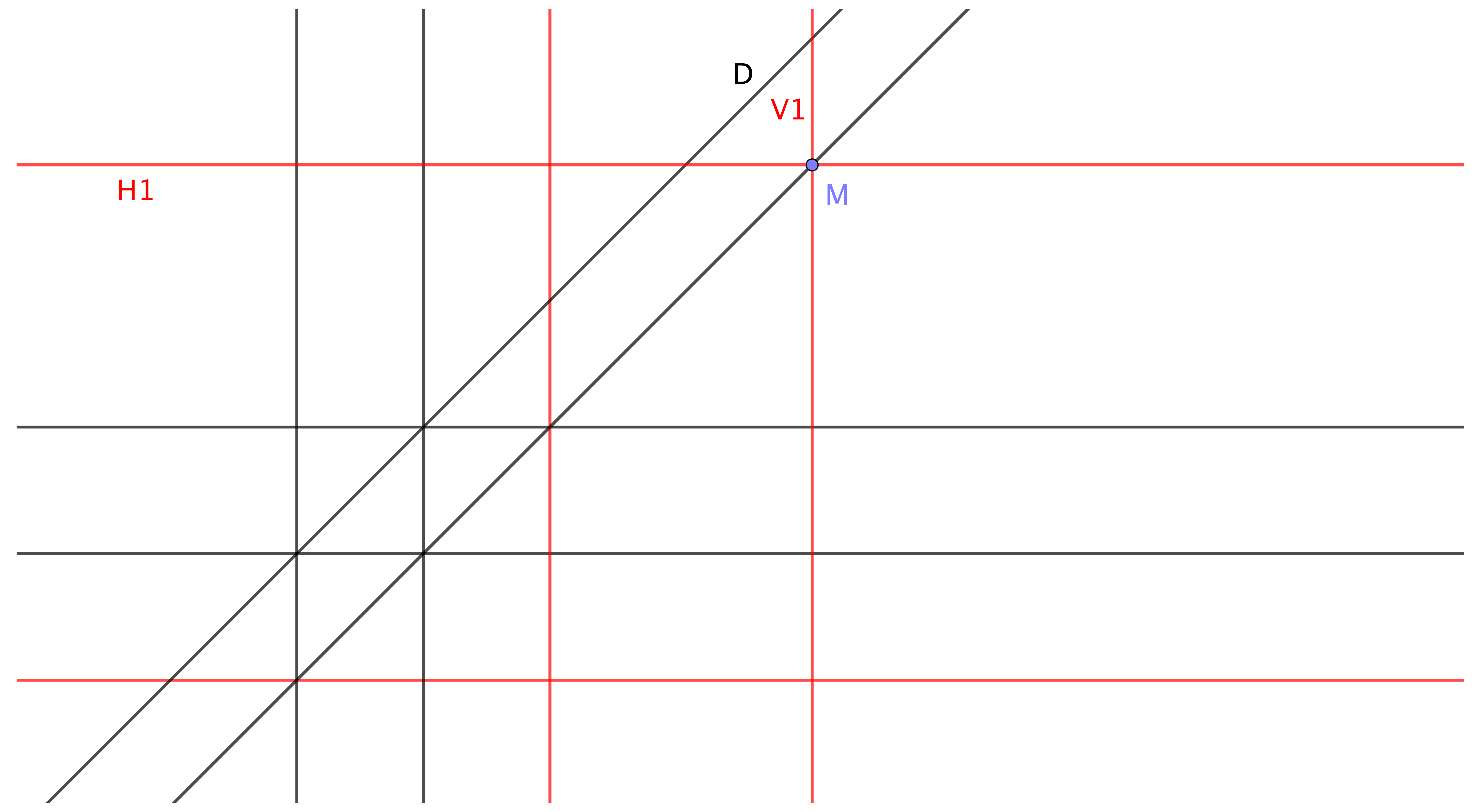}
\caption{}
\label{fig9}
\end{center}
\end{figure}
We claim that in this case, the combinatorics does not allow the jumping point, which we already know belongs to the added line $D$, to be on any another line of  the arrangement.\\
Observe that if we move the lines $H1$ and $V1$, maintaining their direction and the point $M$ as a triple point of the arrangement, we keep the same combinatorics and the jumping point $P$, associated to the obtained arrangement, moves along the line $D$.\\
Let us explain why $P\notin \mathcal{A}$. We must prove that the splitting type is $(5,5)$ for any line of $\mathcal{A}$. Directly the multiplicity of the points given by the other lines of the arrangement determine the splitting type for the red lines in the picture and for the line at infinity; indeed their splitting is $(5,5)$.\\
Let us consider now the black horizontal and vertical lines, where the points defined on them by the other lines of the arrangement have multiplicities, up to  order, $(1,1,2,2,4)$. Consider one of such black lines, which we will denote by $L$, with $\mathbb{C}[x,z]$ as its homogeneous ring of coordinates. The line $L$ contains three multiple points and two simple points.\\
Since $\mathrm{PGL}(2)$ acts transitively on three points we can assume that the point $(1,0)$ has multiplicity $4$, $(0,1)$ and $(1,1)$ have multiplicity $2$. The remaining points $(t_1,1)$ and $(t_2,1)$ are simple points. If $L$ is a jumping line, there exist, according to \cite{WY}, two polynomials $P$ and $Q$ with degree $4$ such that :
$$ z^4|Q, \,\, x^2|P, \,\, (x-z)^2|(P-Q), \,\, (x-t_1z)| (P-t_1z)\,\, \mathrm{and}\,\, (x-t_2z)| (P-t_2z).$$
We obtain that, by direct computation, $Q(x,z)=z^4$, $P(x,z)=x^2(x^2-4xz+4z^2)$, $t_1=\frac{3-\sqrt{5}}{2}$ and $t_2=\frac{3+\sqrt{5}}{2}$.\\ 
But considering this special position of points on $L$, it is impossible to recover the required combinatorics; in particular, it is not possible to find a line, parallel to the main diagonal line, through two triple points of the grid.
Figure \ref{fig10} explains the behaviour of this special arrangement.\\
\begin{figure}[ht!]
\begin{center}
\includegraphics[scale=0.14]{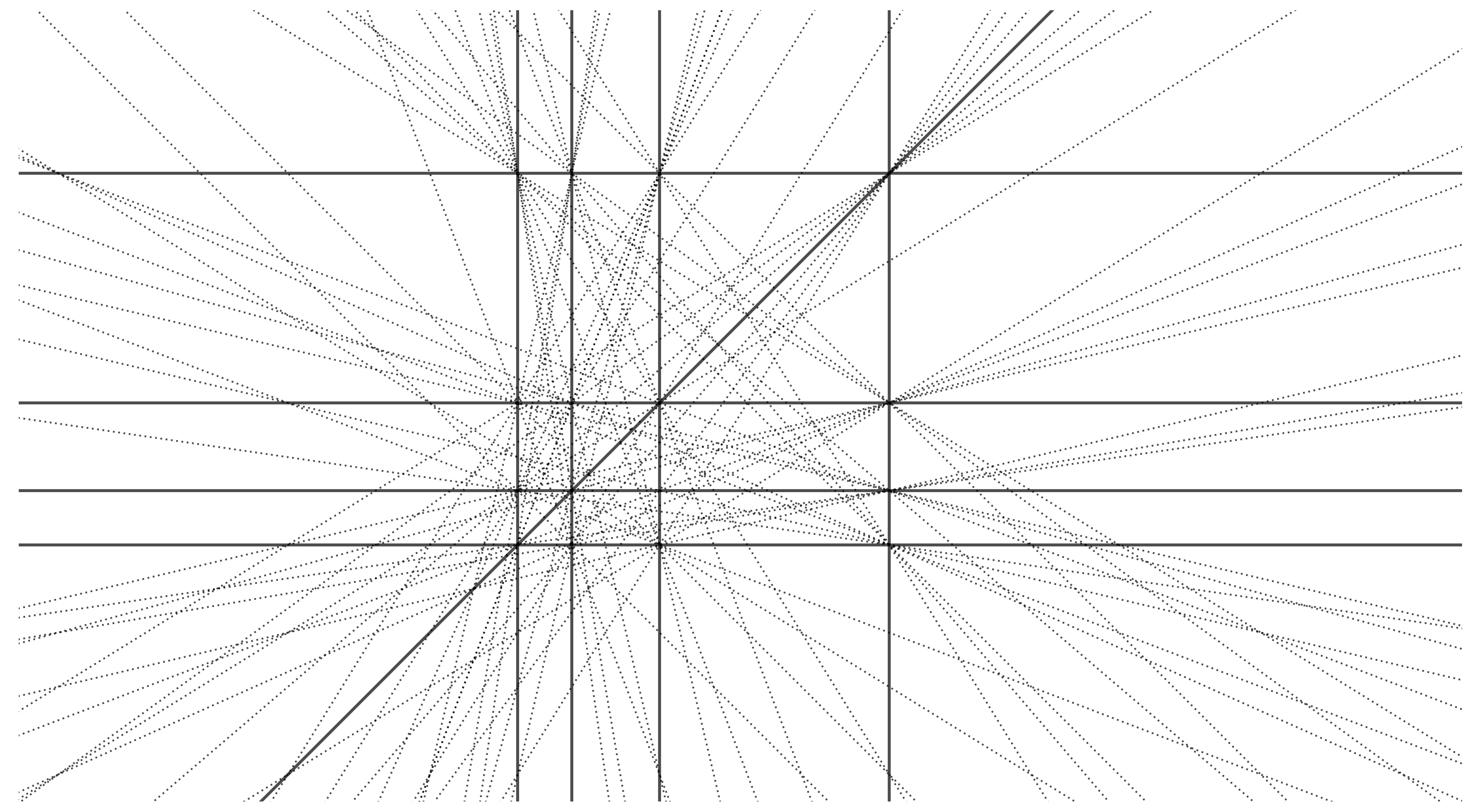}
\caption{}
\label{fig10}
\end{center}
\end{figure}
\end{Example}

The next example shows us that we can find a specific combinatorics for which the jumping point can belong or not to the arrangement. Indeed, we will see that shifting properly the lines of the arrangement, in order to maintain its combinatorics, the jumping point will either be inside or outside the arrangement.
\begin{Example}\label{ex-inout}
Let us consider the arrangements defined by $$C_t:= xyz(x-z)(x-2z)(x-tz)(y-z)(y-2z)(y-(t+1)z)(x-y)(x-y+z)=0.$$\\
If $t=1/2$, the associated nearly free bundle is defined by the matrix
$$
A=[-7x+11y-7z \:\: 14x-134y+161z \:\: 4y^2-7yz]
$$
and the jumping point $P=(17:21:16)$ does not belong to the arrangement $C_{1/2}$.\\
On the contrary, if we take $t=2/3$, then the associated nearly free bundle is defined by the matrix 
$$
A=[-5x+8y-5z \:\: 15x-144y+165z \:\: 6y^2-10yz]
$$
and the jumping point $P=(4:5:4)$ belongs to the arrangement $C_{2/3}$.\\
In Figure \ref{fig8}, the black lines are the ones in common between the two arrangements, while the red ones belong to $C_{2/3}$ and the blue ones to $C_{1/2}$. It can be checked directly from the picture that the two arrangements have the same combinatorics.
\begin{figure}[ht!]
\begin{center}
\includegraphics[scale=0.17]{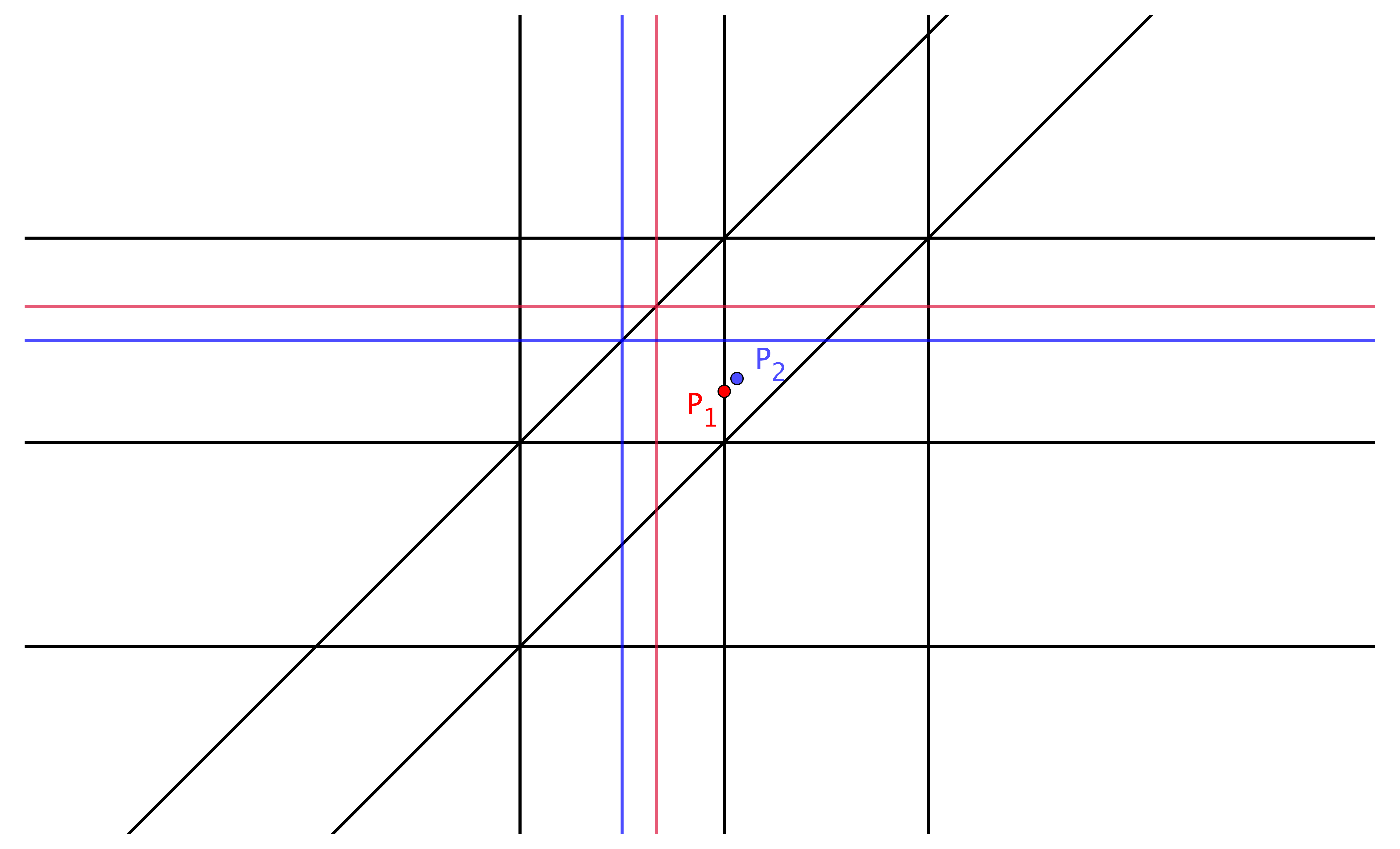}
\caption{}
\label{fig8}
\end{center}
\end{figure}
\end{Example}

\end{document}